\documentclass[letterpaper,12pt]{amsart}

\usepackage[all]{xy}                        %

\usepackage[margin=1in]{geometry}

\CompileMatrices                            

\UseTips                                    

\input xypic
\usepackage[bookmarks=true]{hyperref}       

\usepackage{amssymb,latexsym,amsmath,amscd}
\usepackage{xspace}
\usepackage{color}
\usepackage{graphicx}

\usepackage[utf8]{inputenc}
\usepackage{fourier}
\usepackage{array}
\usepackage{makecell}

\usepackage{array}
\newcolumntype{L}[1]{>{\raggedright\let\newline\\\arraybackslash\hspace{0pt}}m{#1}}
\newcolumntype{C}[1]{>{\centering\let\newline\\\arraybackslash\hspace{0pt}}m{#1}}
\newcolumntype{R}[1]{>{\raggedleft\let\newline\\\arraybackslash\hspace{0pt}}m{#1}}

\usepackage[all]{xy}                        %

\usepackage[margin=1in]{geometry}

\CompileMatrices                            

\UseTips                                    

\input xypic
\usepackage[bookmarks=true]{hyperref}       

\usepackage{amssymb,latexsym,amsmath,amscd}
\usepackage{xspace}
\usepackage{color}
\usepackage{graphicx}
\usepackage{caption}

\usepackage{mathtools}
\usepackage{tikz}
\usetikzlibrary{cd}

\usepackage[utf8]{inputenc}
\usepackage{fourier}
\usepackage{array}
\usepackage{makecell}

\usepackage{array}
\reversemarginpar

\theoremstyle{plain}
\newtheorem{theorem}{Theorem}[section]
\newtheorem*{theorem*}{Theorem}
\newtheorem{proposition}[theorem]{Proposition}
\newtheorem{corollary}[theorem]{Corollary}
\newtheorem{lemma}[theorem]{Lemma}

\theoremstyle{definition}
\newtheorem{definition}[theorem]{Definition}

\newtheorem{remark}[theorem]{Remark}

\newcommand{\enm}[1]{\ensuremath{#1}}          %
\newcommand{\op}[1]{\operatorname{#1}}
\newcommand{\cal}[1]{\mathcal{#1}}

\renewcommand{\bar}[1]{\overline{#1}}

\newcommand{\CC}{\enm{\mathbb{C}}}

\newcommand{\QQ}{\enm{\mathbb{Q}}}
\newcommand{\ZZ}{\enm{\mathbb{Z}}}

\newcommand{\PP}{\enm{\mathbb{P}}}

\newcommand{\EE}{\enm{\mathbb{E}}}

\newcommand{\Aa}{\enm{\cal{A}}}

\newcommand{\Ee}{\enm{\cal{E}}}
\newcommand{\Ff}{\enm{\cal{F}}}
\newcommand{\Gg}{\enm{\cal{G}}}

\newcommand{\Ll}{\enm{\cal{L}}}

\newcommand{\Oo}{\enm{\cal{O}}}
\newcommand{\Pp}{\enm{\cal{P}}}

\newcommand{\Ss}{\enm{\cal{S}}}
\newcommand{\Tt}{\enm{\cal{T}}}

\renewcommand{\phi}{\varphi}
\renewcommand{\theta}{\vartheta}
\renewcommand{\epsilon}{\varepsilon}

\newcommand{\Pic}{\op{Pic}}

\newcommand{\Hom}{\op{Hom}}

\newcommand{\tensor}{\otimes}         

\renewcommand{\to}[1][]{\xrightarrow{\ #1\ }}

\newcommand{\into}{\hookrightarrow}

\renewcommand{\a}{\alpha}

\newcommand{\old}[1]{}

\usepackage{mathtools}

\makeatletter
\def\dasharrowfill@#1#2#3#4{%
        $\m@th
        \thickmuskip0mu
        \medmuskip\thickmuskip
        \thinmuskip\thickmuskip
        \relax
        #4#1\mkern2mu
        \xleaders\hbox{$#4\mkern2mu#2\mkern2mu$}\hfill
        \mkern2mu
        #3$%
}

\def\dashleftarrowfill@{\dasharrowfill@\leftarrow\relbar\relbar}
\def\dashrightarrowfill@{\dasharrowfill@\relbar\relbar\rightarrow}
\def\dashleftrightarrowfill@{\dasharrowfill@\leftarrow\relbar\rightarrow}
\def\dashLeftarrowfill@{\dasharrowfill@\Leftarrow\Relbar\Relbar}
\def\dashRightarrowfill@{\dasharrowfill@\Relbar\Relbar\Rightarrow}
\def\dashLeftrightarrowfill@{\dasharrowfill@\Leftarrow\Relbar\Rightarrow}

\providecommand*\xdashleftarrow[2][]{%
  \ext@arrow 0055{\dashleftarrowfill@}{#1}{#2}}
\providecommand*\xdashrightarrow[2][]{%
  \ext@arrow 0055{\dashrightarrowfill@}{#1}{#2}}
\providecommand*\xdashleftrightarrow[2][]{%
  \ext@arrow 0055{\dashleftrightarrowfill@}{#1}{#2}}
\providecommand*\xdashLeftarrow[2][]{%
  \ext@arrow 0055{\dashLeftarrowfill@}{#1}{#2}}
\providecommand*\xdashRightarrow[2][]{%
  \ext@arrow 0055{\dashRightarrowfill@}{#1}{#2}}
\providecommand*\xdashLeftrightarrow[2][]{%
  \ext@arrow 0055{\dashLeftrightarrowfill@}{#1}{#2}}
\makeatother

\begin{document}

\title[Virtual Poincar\'e polynomial of moduli space of semistable sheaves]{Virtual Poincar\'e polynomial of moduli space \\of semistable sheaves of rank two on reducible curves}

\author{Sukmoon Huh, Dongsun Lim and Sang-Bum Yoo}

\address{Sungkyunkwan University, Suwon 440-746, Korea}
\email{sukmoonh@skku.edu}

\address{Sungkyunkwan University, Suwon 440-746, Korea}
\email{imdong0854@gmail.com}

\address{Gongju National University of Education, Gongju 32553, Korea}
\email{sbyoo@gjue.ac.kr}

\thanks{SH is supported by the National Research Foundation of Korea(NRF) grant funded by the Korea government(MSIT) (No. RS-2023-00208874).}



\begin{abstract}
The main purpose of this paper is to give an explicit description of the moduli space of semistable sheaves of rank two on a stable curve $C$ obtained by gluing two smooth curves at a point. We prove that the moduli space is irreducible and birational to a projective bundle over the moduli space of stable vector bundles on each component curve, independently of the choice of polarization. As an application, we compute the virtual Poincar\'e polynomial of the moduli space. 
\end{abstract}

\maketitle

\section{Introduction}

The classification of semistable sheaves on singular or reducible curves is a central problem in moduli theory. It aims to understand how vector bundles degenerate and how the associated moduli spaces reflect the geometry of the underlying curve. In the case of nodal curves, especially those that are reducible, the classical definition of (semi)stability must be refined by introducing a polarization. In such a setting, the structure of the moduli space is determined by the restrictions of sheaves to each irreducible component and by the gluing conditions at the nodes.

A foundational construction in the study of semistable sheaves on singular curves was given by D.~S. Nagaraj and C.~S. Seshadri in \cite{nagaraj-seshadri,seshadri}. They extended the classical moduli theory to nodal curves by constructing moduli spaces of torsion-free sheaves via geometric invariant theory. Moreover, further works by L.~Caporaso, C.~Simpson, and others generalized this construction to include compactified Picard schemes and moduli of coherent sheaves with suitable depth conditions; see \cite{c1,c2,c3,simpson}.

In \cite{teixidor,teixidor2}, M.~Teixidor i Bigas further analyzed the moduli of semistable sheaves on tree-like curves and showed that for a generic polarization, the moduli space may decompose into several irreducible components, each determined by a distinct multi-degree. This illustrates the intricate structure of the moduli space and its dependence on the chosen polarization. Within this framework, S.~Brivio and F.~F. Favale studied the moduli space of rank two semistable sheaves on a reducible nodal curve consisting of two smooth components intersecting transversely at a node in \cite{brivio-favale}. For a suitable choice of polarization, they proved that the moduli space admits a birational description as a projective bundle over the product of the moduli spaces of semistable rank two vector bundles on each component. Moreover, they demonstrated that this space arises as a degeneration of the moduli space defined over a smooth curve, illustrating how the geometry of the singular curve is reflected in the structure of the associated moduli space.

In this article, we study the moduli space $\mathbf{U}_C(\mathbf{w},2,0)$ of $S_\mathbf{w}$-equivalence classes of $\mathbf{w}$-semistable sheaves of depth one on a reducible curve $C = C_1 \cup C_2$ with rank two and Euler characteristic zero, where $C_1$ and $C_2$ are smooth projective curves of genus $g_1, g_2 \ge 1$, meeting transversely at a node $p \in C$. For a coherent sheaf $\Ee$ on $C$ with depth one, one can consider a sheaf $\Ee_i$ defined as the restriction of $\Ee$ to $C_i$ modulo torsion, from which one can consider the multi-rank and multi-characteristic of $\Ee$. In particular, S.~Brivio and F.~F. Favale \cite{brivio-favale} focused on sheaves with coprime rank and degree, restricting to pure rank cases. In contrast, we describe sheaves with multi-rank, including non-coprime cases. First we investigate the classification of such sheaves with lower multi-ranks and possible multi-characteristics; see Propositions \ref{propr0}, \ref{pprop22}, and \ref{p555}. This gives a building block for the Jordan-Hölder filtrations of the strictly $\mathbf{w}$-semistable sheaf $\Ee$ in $\mathbf{U}_C(\mathbf{w},2,0)$. Indeed, as shown in Propositions \ref{(2,0)(0,1)} and \ref{(2,0)(0,2)}, the loci of strictly $\mathbf{w}$-semistable sheaves form a single component in $\mathbf{U}_C(\mathbf{w},2,0)$.


To describe this moduli space more concretely, we show that for any polarization $\mathbf{w}$, the space $\mathbf{U}_C(\mathbf{w},2,0)$ is irreducible and admits an explicit description obtained from a birational morphism from a projective bundle over $\mathbf{M}_{C_1}(2,2g_1-1) \times \mathbf{M}_{C_2}(2,2g_2-1)$, where $\mathbf{M}_{C_i}(r,d)$ denotes the moduli space of semistable vector bundles of rank $r$ and degree $d$ on $C_i$. More precisely, every $\mathbf{w}$-semistable sheaf in $\mathbf{U}_C(\mathbf{w},2,0)$ can be described as the kernel of a surjection $\Ee_1 \oplus \Ee_2 \longrightarrow \Oo_p^{\oplus 2}$, where $\Ee_i \in \mathbf{M}_{C_i}(2,2g_i-1)$. This construction allows us to define a vector bundle $\EE$ over $\mathbf{M}_{C_1}(2,2g_1-1) \times \mathbf{M}_{C_2}(2,2g_2-1)$, using the Poincar\'e bundle on $\mathbf{M}_{C_i}(2,2g_i-1)$ for each $i$, whose fibre over $(\Ee_1, \Ee_2)$ is $\mathrm{Hom}(\Ee_{1,q_1}, \Ee_{2,q_2})$. It gives rise to a rational map
\[
\Phi: \quad \PP(\EE) \xdashrightarrow{\hspace*{1 cm}} \mathbf{U}_C(\mathbf{w},2,0).
\]
As proven in Proposition \ref{sub}, this map is in fact a morphism, and the points of $\PP(\EE)$ parametrize the classes of $\mathbf{w}$-semistable sheaves. In fact, this morphism is birational. 

\begin{theorem}\label{main}
The map $\Phi : \PP (\EE) \rightarrow \mathbf{U}_C(\mathbf{w}, 2, 0)$ is a birational morphism, which is an isomorphism if and only if $g=2$.
\end{theorem}

\noindent The strictly $\mathbf{w}$-semistable locus $\mathbf{Q}$ in $\PP(\EE)$ is characterized by rank-one linear maps between the fibres and defines a divisor that is irreducible and of codimension one. When each $C_i$ is an elliptic curve, this divisor admits the structure of a $(\PP^1 \times \PP^1)$-fibration over the base $\mathbf{M}_{C_1}(2,1) \times \mathbf{M}_{C_2}(2,1)$, with fibres corresponding to smooth quadrics in $\PP^3$; see Remark \ref{QQ}. The general description of the birational morphism $\Phi$ on $\mathbf{Q}$ can be given in terms of the Hecke curves in Remark \ref{Poin:ell}, and this enables explicit computations of virtual invariants, such as the virtual Poincaré polynomial.

\begin{theorem}\label{virtu}
The virtual Poincar\'e polynomial of $\mathbf{U}_C(\mathbf{w},2,0)$ can be expressed as
\[
\mathrm{P}(\mathbf{U}_C(\mathbf{w},2,0))
= (1+t)^{2g}\Big((1+t)^{2g-4} S_{g_1}(t)\,S_{g_2}(t)\,(-t^2+t^6) + U_{g_1}(t)\,U_{g_2}(t)\Big),
\]
where $S_g(t), U_g(t) \in \mathbb{Z}[t]$ are given by
\begin{align*}
S_g(t) &= \frac{(1-t+t^2)^{2g}-t^{2g}}{(1-t)^2(1+t^2)}, \\
\displaystyle U_g(t) &= \frac{2(1+t)^{2g-2}\big((1-t+t^2)^{2g}-t^{2g+2}\big) - t^2\big( (1+t)^{2g-2}(1-t^4) - (1-t)^{2g+2}\big)}{2(1-t)^2(1+t^2)}.    
\end{align*}
\end{theorem}

Motivated by the classical description of vector bundle moduli on smooth curves, one can interpret the moduli space $\mathbf{U}_C(\mathbf{w},2,0)$ through a degeneration process. Based on the foundational work of M.~S. Narasimhan and S.~Ramanan, who described in \cite{narashimhan-ramanan} the moduli space $\mathbf{M}_{C'}(2,0)$ of semistable vector bundles of rank two and degree zero on a smooth projective curve $C'$ of genus two, we consider a degeneration of $C'$ to a reducible nodal curve $C = C_1 \cup C_2$. In this setting, the moduli space $\mathbf{U}_C(\mathbf{w},2,0)$ arises naturally as a degeneration of $\mathbf{M}_{C'}(2,0)$ that is isomorphic to a $\PP^3$-bundle over $\Pic^0(C')$. Note that we have a degeneration of $\mathrm{Pic}^0(C')$ to $\mathrm{Pic}^0(C) \cong \mathrm{Pic}^0(C_1) \times \mathrm{Pic}^0(C_2)$; see \cite{lopez-martin}. The comparison between the two moduli spaces $\mathbf{U}_C(\mathbf{w}, 2, 0)$ and $\mathbf{M}_{C'}(2,0)$ via their virtual Poincar\'e polynomials is given in Corollary \ref{comparison}. 

On the other hand, M.~Teixidor i Bigas showed that the moduli space may have more than one irreducible component for general polarization; see \cite{teixidor,teixidor2}. It could be worthy to point out that this is not the case for any polarization in our setting; indeed, $\mathbf{U}_C(\mathbf{w},2,0)$ is irreducible for any $\mathbf{w}$.

The paper is organized as follows. In Section 2, we recall the definition of $\mathbf{w}$-stability and basic properties of depth-one sheaves on reducible curves. In Section 3, we describe and classify all semistable sheaves in $\mathbf{U}_C(\mathbf{w},2,0)$ and construct a birational morphism from the projective bundle $\PP(\EE)$. In Section 4, we compute the virtual Poincar\'e polynomial of the moduli space $\mathbf{U}_C(\mathbf{w}, 2, 0)$.

We would like to thank Young-Hoon Kiem and Insong Choe for explaining to us the Hecke cycles and the virtual Poincar\'e polynomial of the moduli space of semistable vector bundles on curves. 


\section{Preliminaries}
Let $C$ be a stable curve of arithmetic genus $g \geq 2$ over an algebraically closed field of characteristic zero. Then $C$ is a complete, connected curve with at worst nodal singularities and only finitely many automorphisms. Write $C=C_1 \cup \dots \cup C_m$ with each $C_i$ is an irreducible component of $C$ with genus $g_i \ge 1$. The set of singular points $\mbox{Sing}(C)$ can be written as a disjoint union of two subsets:
\[
\mbox{Sing}(C) = \mbox{Sing}_1(C)\sqcup \mbox{Sing}_2(C),
\]
where $\mbox{Sing}_j(C)$ is the set of nodes of $C$ lying exactly on $j$ irreducible components of $C$. Following the approach of C.~S. Seshadri in \cite{seshadri}, we consider depth-one sheaves to describe the moduli space of sheaves on $C$. For a coherent sheaf $\Ee$ on $C$ with {\it depth one}, meaning that $\dim \mathrm{Supp}(\Ff)=1$ for any subsheaf $\Ff\subset \Ee$, we define the {\it multi-rank} and {\it multi-degree} of $\Ee$ by
\[
\mathrm{rank}(\Ee) =\vec{r} = (r_1, \dots, r_m) \in \ZZ_{\ge 0}^{\oplus m},\quad \mathrm{deg}(\Ee) =\vec{d}= (d_1, \dots, d_m) \in \ZZ^{\oplus m},
\]
where $r_i:=\mathrm{rank}(\Ee_{|C_i})$ and $d_i:=\deg (\Ee_{|C_i})$. Here, $\Ee_{|C_i}:=\left(\Ee \tensor \Oo_{C_i}\right)/(\text{torsion})$ is called the {\it restriction} of $\Ee$ to $C_i$. We also define the {\it multi-characteristic} $\vec{\chi}=(\chi_1, \dots, \chi_m) \in \ZZ^{\oplus m}$ of $\Ee$ with $\chi_i:=\chi(\Ee_i)=d_i+r_i(1-g_i)$. In case $r_1=\dots=r_m=r$, we say that $\Ee$ has {\it pure rank} $r$ and write that $\mathrm{rank}(\Ee)=\vec{r}=r$. Similarly, we may define the {\it pure degree} and {\it pure characteristic}.

\begin{definition}\label{def1}
Let $\mathbf{w} = (w_1, \dots, w_m) \in \QQ_{>0}^{\oplus m}$ be an $m$-tuple of positive rational numbers with $\sum_{i}w_i = 1$, called a {\it polarization} of $C$. If $\Ee$ is a depth-one sheaf of multi-rank $(r_1, \dots, r_m)$, the {\it$\mathbf{w}$-rank} is defined to be ${\rm rank}_{\mathbf{w}}(\Ee)=\sum_{i=1}^{m}w_{i}r_i \in \QQ$, and the {\it $\mathbf{w}$-slope} of $\Ee$ on $C$ is defined to be
\[
\mu_{\mathbf{w}}(\Ee)=\frac{\chi(\Ee)}{\mathrm{rank}_{\mathbf{w}}(\Ee)}\in \QQ.
\]
Then the sheaf $\Ee$ is said to be {\it $\mathbf{w}$-stable} (resp. {\it $\mathbf{w}$-semistable}) if the inequality 
\[
\mu_{\mathbf{w}}(\Ff) < \text{(resp. $\le$) } \mu_{\mathbf{w}}(\Ee)
\]
holds for any proper subsheaf $\Ff$ of $\Ee$.
\end{definition}

\begin{remark}\label{rem333}
Assume that $C=C_1 \cup C_2$ with $\mbox{Sing}_1(C)=\emptyset$ and $\mbox{Sing}_2(C)=\{p\}$. Let $\Ee$ be a $\mathbf{w}$-semistable sheaf of depth one on $C$, meaning that it has no associated points of dimension zero, or equivalently, the singularities of $\Ee$ are at most of codimension one. We assume that $\mathrm{rank}(\Ee)=(r_1,r_2)$, $\chi(\Ee)=0$, and let $\Ee_i:=\Ee_{|C_i}$ be the restriction of $\Ee$ to $C_i$. Consider the exact sequence for $\{i,j\} = \{1,2\}$:
\[
0 \rightarrow \Oo_{C_j}(-C_i) \rightarrow \Oo_C \rightarrow \Oo_{C_i} \to 0.
\]
This induces a surjection:
\[
\Ee \stackrel{\rho_i} {\longrightarrow} \Ee_i \rightarrow 0,
\]
from which we deduce that $\chi(\Ee_i) =\chi_i \geq 0$ due to the $\mathbf{w}$-semistability of $\Ee$. Now, consider the following exact sequence:
\[
0 \rightarrow \Ee \stackrel{(\rho_1,\rho_2)} {\xrightarrow{\hspace*{1.1cm}}}  \Ee_1 \oplus \Ee_2  \stackrel{\theta_{\tau}} {\longrightarrow} \Tt \rightarrow 0, \tag{$\dagger$}
\]
where $\Tt$ is isomorphic to $\Oo_p^{\oplus \tau}$, with $\tau \in \{0,1,2,\dots,r_0\}$ with $r_0:=\mathrm{min}\{r_1, r_2\}$; see \cite[Proposition 2.2]{nagaraj-seshadri}. We denote by $\iota_i: \Ee_i \hookrightarrow \Ee_1 
 \oplus \Ee_2 \twoheadrightarrow \Oo_p^{\oplus \tau}$ the natural composition map. Conversely, for a triple $\zeta := (\Ee_1, \Ee_2, \theta_{\tau})$ where each $\Ee_i$ is a vector bundle on $C_i$, one can define a depth-one sheaf $\Ee_{\zeta} := \ker(\theta_{\tau})$ on $C$. We now consider the normalization map $\nu: C_1 \sqcup C_2 \rightarrow C$, which identifies $q_1$ and $q_2$ as $\nu^{-1}(p)=\{q_1,q_2\}$. For a coherent sheaf $\Ee$ on $C$ with depth one, its stalk at the node $p$ satisfies
\[
\Ee_p \cong \Oo_{C,p}^{\oplus a} \oplus \Oo_{q_1}^{\oplus b} \oplus \Oo_{q_2}^{\oplus c}
\]
for some integers $a, b, c$. Here, $a$ represents the free rank at $p$, while $b$ and $c$ describe torsion ranks on $C_1$ and $C_2$, respectively. We define the type of $\Ee$ at $p$ as $\delta(\Ee) := (a, b, c)$. By tensoring ($\dagger$) with $\Oo_p$, we obtain the exact sequence:
\[
0 \rightarrow {\rm{Tor}}_C^1(\Tt, \Oo_p) \rightarrow \Ee_p \rightarrow \Ee_{1,q_1} \oplus \Ee_{2,q_2} \rightarrow \Tt \rightarrow 0.
\]
Since $p$ is a singular point, ${\rm{Tor}}_C^1(\Tt, \Oo_p)$ is not necessarily trivial. Moreover, there is a canonical isomorphism
\[
{\rm{Tor}}_C^1(\Tt, \Oo_p) \cong \Tt \otimes {\rm{Tor}}^1(\Oo_p, \Oo_p).
\]
Since $p$ is a nodal point, we have ${\rm{Tor}}^1(\Oo_p, \Oo_p) \cong \Oo_p$, and so ${\rm{Tor}}_C^1(\Tt, \Oo_p) \cong \Tt$. Therefore, we conclude that
\[
\Ee_{1,q_1} \cong \Oo_{C,p}^{\oplus a} \oplus \Oo_{q_1}^{\oplus b}, \quad \Ee_{2,q_2} \cong \Oo_{C,p}^{\oplus a} \oplus \Oo_{q_2}^{\oplus c},
\]
and $\Tt \cong \Oo_p^{\oplus \tau}$. In particular, the rank relations are given by $r_1 = a + b$ and $r_2 = a + c$. Moreover, the Euler characteristic satisfies
\[
\chi_1 + \chi_2 = \tau,
\]
since $\chi(\Ee) = 0$ and $\chi_i \geq 0$. 
\end{remark}

\begin{lemma}\label{commute}
With the notations in Remark \ref{rem333}, let $\Ff$ be a saturated subsheaf of $\Ee$ in depth one, i.e., the quotient $\Ee/\Ff$ is of depth one. Then, the following exact sequence 
\[
0 \rightarrow \Ff \rightarrow \Ff_1 \oplus \Ff_2 \rightarrow \Tt' \rightarrow 0
\]
with $\Ff_i:= \Ff_{|C_i}$ and $\Tt'$ supporting the node $p$, fits into the commutative diagram
\[
\begin{tikzcd}
& 0 \arrow[d] & 0 \arrow[d] & 0 \arrow[d] & & \\
0 \arrow[r] & \Ff \arrow[r] \arrow[d] & \Ff_1 \oplus \Ff_2 \arrow[r] \arrow[d] & \Tt' \arrow[r] \arrow[d] & 0 & \\
0 \arrow[r] & \Ee \arrow[r] \arrow[d] & \Ee_1 \oplus \Ee_2 \arrow[r] \arrow[d] & \Tt \cong \Oo_p^{\oplus \tau} \arrow[r] \arrow[d] & 0 \\
0 \arrow[r] & \Ee/\Ff \arrow[r] \arrow[d] & (\Ee_1/\Ff_1) \oplus (\Ee_2/\Ff_2) \arrow[r] \arrow[d] & \Oo_p^{\oplus \tau}/\Tt' \arrow[r] \arrow[d] & 0 \\
& 0 & 0 & 0. & 
\end{tikzcd}
\]
\end{lemma}

\begin{proof}
From the injective map $\a_i: \Ff_i \rightarrow \Ee_i$ induced from $\Ff \rightarrow \Ee$, one can define an injective map $\Ff_1 \oplus \Ff_2 \hookrightarrow \Ee_1 \oplus \Ee_2$ diagonally. Thus its cokernel is isomorphic to $(\Ee_1/\Ff_1) \oplus (\Ee_2/\Ff_2)$, and we get a map 
\[
\Ee/\Ff \to (\Ee_1/\Ff_1) \oplus (\Ee_2/\Ff_2).
\]
Since the composition $\Ff \rightarrow \Ff_1\oplus \Ff_2 \rightarrow \Ee_1\oplus \Ee_2 \rightarrow \Tt$ is trivial, we get a map $g:\Tt' \rightarrow \Tt$. By the Snake Lemma, $\ker(g)$ injects into $\Ee/\Ff$, and so we get $\ker(g) \cong 0$ by the assumption on $\Ff$. Then we get a surjection from $(\Ee_1/\Ff_1) \oplus (\Ee_2/\Ff_2)$ to $\Oo_p^{\oplus \tau}/\Tt'$, and the assertion follows.
\end{proof}

\begin{remark}
For each $\mathbf{w}$-semistable sheaf $\Ee$ of depth one on $C$, there exists a finite filtration of sheaves of depth one on $C$:
\[
0=\Ee^0 \subsetneq \Ee^1 \subsetneq \Ee^2 \subsetneq \cdots \subsetneq \Ee^l = \Ee
\]
such that each quotient $\Ee^i\big/ \Ee^{i-1}$ is a $\mathbf{w}$-stable sheaf of depth one on $C$ with a polarized slope 
\[
\mu_{\mathbf{w}}(\Ee^i \big/\Ee^{i-1})=\mu_{\mathbf{w}}(\Ee).
\]
This filtration is analogous to the Jordan-Hölder filtration in classical stability theory and plays a crucial role in defining the associated graded sheaf and $S_\mathbf{w}$-equivalence.
\end{remark}

\begin{definition}
The sheaf 
\[
\mathfrak{gr}_{\mathbf{w}}(\Ee)=\bigoplus_{i=1}^{l} \left(\Ee^{i}\big/\Ee^{i-1}\right)
\]
is called {\it{graded sheaf associated to}} $\Ee$. Two $\mathbf{w}$-semistable sheaves $\Ee$ and $\Ee'$ of depth one on $C$ are {\it{$S_\mathbf{w}$-equivalent}} if $\mathfrak{gr}_{\mathbf{w}}(\Ee) \cong \mathfrak{gr}_{\mathbf{w}}(\Ee')$. When $C$ is irreducible, i.e., $m=1$, we set $\mathbf{w}=(1)$ and $\mathfrak{gr}(\Ee) = \mathfrak{gr}_{\mathbf{w}}(\Ee)$ by convention. 
\end{definition}

\begin{definition}
We denote by $\mathbf{U}_C(\mathbf{w}, \vec{r}, \chi)$ the moduli space of $S_\mathbf{w}$-equivalence classes of depth one, $\mathbf{w}$-semistable sheaves on $C$ with multi-rank $\vec{r} = (r_1, \dots, r_m)$ and Euler characteristic $\chi$; see \cite[Chapter 7]{seshadri}. For a multi-characteristic $\vec{\chi}=(\chi_1, \cdots, \chi_m)$, we denote by $\mathbf{U}_C(\mathbf{w},\vec{r},\chi)_{\vec{\chi}}$ the irreducible component of $\mathbf{U}_C(\mathbf{w},\vec{r},\chi)$ that contains the classes $[\Ee]$ with $\chi_i=\chi(\Ee_i)$ for each $i$. Then we have
\begin{equation}\label{decomp}
\mathbf{U}_C(\mathbf{w}, \vec{r}, \chi) = \bigcup_{\vec{\chi}\in \ZZ^{\oplus m}}\mathbf{U}_C(\mathbf{w}, \vec{r}, \chi)_{\vec{\chi}}.
\end{equation}
\end{definition}

\begin{proposition}\label{texi}\cite[proposition 2.1]{teixidor2}
In the case $m=2$, the moduli space $\mathbf{U}_C(\mathbf{w}, r, \chi)$ is connected and has $r$ irreducible components with respect to generic polarization $\mathbf{w}$. Moreover, each component has dimension $r^2(g-1)+1$. 
\end{proposition}


\begin{remark}
For a polarization $\mathbf{w}=(w_1, \dots, w_m)$, the moduli space $\mathbf{U}_C(\mathbf{w}, \vec{r}, \chi)$ parametrizes $\mathbf{w}$-(semi)stable sheaves on $C$ of multi-rank $\vec{r}$ and Euler characteristic $\chi$. When $C$ is irreducible, i.e., $m=1$, we set $\mathbf{w} = (1)$ by convention. In this case, $\mathbf{U}_C(\mathbf{w}, r, \chi)$ is denoted by $\mathbf{M}_C(r,d)$, where $d = \chi - r(1-g)$. For smooth irreducible curves, C.~S. Seshadri’s GIT construction establishes that $\mathbf{M}_C(r,d)$ is an irreducible, projective scheme; see \cite{seshadri}. The dimension of $\mathbf{M}_C(r,d)$ is given by
\begin{equation}
\dim \mathbf{M}_{C}(r,d) = \begin{cases} 
r^2(g-1)+1, & g \ge 2, \\
\gcd(r,d), & g=1.\end{cases}
\end{equation}
\noindent For the general case, one can refer to \cite{simpson}.
\end{remark}

\begin{remark}\label{genus}
Let $C$ be a smooth, irreducible, projective curve of genus $g \ge 1$, and let $\Ee$ be a vector bundle of rank $2$ and degree $d$ on $C$. By tensoring with a line bundle, we may assume $\deg(\Ee) \in \{0,1\}$.

\noindent\hangindent=2.5em (a-1) Suppose that $\Ee$ is a semistable vector bundle with $\deg(\Ee)=0$. If $\Ee$ is strictly semistable, then it fits into a short exact sequence
\[
0 \to \Ll \to \Ee \to \Ll' \to 0,
\]
for some $\Ll, \Ll' \in \Pic^0(C)$. It gives the Jordan-Hölder filtration of $\Ee$, so the $S$-equivalence class of $\Ee$ is $[\Ll \oplus \Ll']$, which is a point in $\mathrm{Sym}^2(\Pic^0(C))$. Indeed, the complement of the locus $\mathbf{M}_C^0(2,0)$ of stable vector bundles is identified with $\mathrm{Sym}^2(\Pic^0(C))$. When $g = 1$, the locus $\mathbf{M}_C^0(2,0)$ is empty, and so we get $\mathbf{M}_C(2,0) \cong \mathrm{Sym}^2(C)$; see \cite[Theorem 1]{tu}. For higher genus $g \geq 2$, the locus $\mathrm{Sym}^2(\Pic^0(C))$ forms a proper closed subset of $\mathbf{M}_C(2,0)$. Indeed, it is the singular locus for $g\ge 3$, while $\mathbf{M}_C(2,0)$ is smooth for $g=2$. 

\noindent\hangindent=2.5em (a-2) If $\Ee$ is semistable with $\deg (\Ee)=1$, then $\Ee$ is automatically stable. Note that there exists a natural determinant map $\det : \mathbf{M}_C(2,1) \rightarrow \Pic^1(C)$. It is an isomorphism for $g=1$, in which case we get $\mathbf{M}_C(2,1) \cong C$.

\noindent\hangindent=2.5em (b) \hspace{0.5em} If $\Ee$ is not semistable, then there exists a destabilizing line subbundle $\Ll \in \Pic^{\ell}(C)$ with $\ell \geq 1$ with an exact sequence
\[
0 \rightarrow \Ll \rightarrow \Ee \rightarrow \Ll' \rightarrow 0,
\]
for some $\Ll' \in \Pic^{\ell'}(C)$. One has $\mathrm{Ext}_C^1(\Ll',\Ll) \cong \mathrm{H}^0(C, \omega_C \otimes \Ll' \otimes \Ll^{\vee})^{\vee}$ by the Riemann-Roch, and so 
\[
\dim\left(\mathrm{Ext}_C^1(\Ll',\Ll)\right)= \max\{0,2g-2-2\ell+d\}, \quad \text{where} \; d=\deg(\Ee).
\]
Hence, the extension splits when $\ell \ge g+d-1$. In particular, if $C$ is an elliptic curve, every non-semistable rank $2$ vector bundle of degree $d$ is isomorphic to a direct sum $\Ll \oplus \Ll'$ with $\Ll \in \Pic^{\ell}(C)$, $\Ll' \in \Pic^{d-\ell}(C)$. The isomorphism classes of such bundles are parametrized by $C \times C$, modulo the $S_2$-action when $d=0$.
\end{remark}

\begin{remark}\label{propsemi}
For a sheaf $\Ee$ of depth one on $C=C_1 \cup \dots \cup C_m$, we have the following properties on $\mathbf{w}$-stability; see \cite[Corollary 8 and Proposition 12 in Ch. 7]{seshadri}.
    
\noindent\hangindent=1.5em  (a) Let $\Ll$ be a line bundle on $C$ such that $w_i\deg(\Ll_j)=w_j\deg(\Ll_i)$ for all $1 \leq i,j \leq m$, where $\Ll_i:=\Ll_{|C_i}$. Then, $\Ee$ is $\mathbf{w}$-stable (resp. $\mathbf{w}$-semistable) if and only if $\Ee \otimes \Ll$ is $\mathbf{w}$-stable (resp. $\mathbf{w}$-semistable). Note that $\Ee \otimes \Ll$ is also a sheaf of depth one.


\noindent\hangindent=1.5em  (b) $\Ee$ is $\mathbf{w}$-stable (resp. $\mathbf{w}$-semistable) if and only if, for any nontrivial quotient sheaf $\Gg$ of depth one on $C$, we have:
\[
\mu_{\mathbf{w}}(\Gg) > \mu_{\mathbf{w}}(\Ee) \quad (\mathrm{resp.} \geq)
\]

\noindent\hangindent=1.5em  (c) Let $\Ee$ and $\Ff$ be $\mathbf{w}$-semistable sheaves of depth one on $C$ with $\mu_{\mathbf{w}}(\Ee)=\mu_{\mathbf{w}}(\Ff)$. For a homomorphism $h:\Ee \rightarrow \Ff$, the sheaves $\mathrm{ker}(h)$ and $\mathrm{coker}(h)$ are of depth one and remain $\mathbf{w}$-semistable with
\[
\mu_{\mathbf{w}}(\mathrm{coker}(h))=\mu_{\mathbf{w}}(\mathrm{ker}(h))=\mu_{\mathbf{w}}(\Ee).
\]
\end{remark}




\section{Description of moduli space}\label{sec3}
In this section we will focus on the case when $C= C_1 \cup C_2$ with each $C_i$ a smooth curve of genus $g_i\ge 1$ and $\mbox{Sing}_2(C)=\{p\}$. In particular, $C$ is a stable curve of genus $g=g_1+g_2$. The main goal of this section is to describe explicitly the moduli space $\mathbf{U}_C(\mathbf{w}, 2, 0)$ of $\mathbf{w}$-semistable sheaves of pure rank $2$ on $C$ with $\chi(\Ee)=0$. More generally, let $\Ee$ be a $\mathbf{w}$-semistable sheaf on $C$ with $\mathrm{rank}(\Ee)=(r_1, r_2)$ and $\chi(\Ee)=0$. In the corresponding exact sequence ($\dagger$) to $\Ee$, let $\Aa_1$ be the image of the composition map $\Ee_1 \rightarrow \Ee_1\oplus \Ee_2 \rightarrow \Tt$. Then for its residual torsion sheaf $\Aa_2 \subset \Tt$, we get a surjection $\Ee_2 \twoheadrightarrow \Aa_2$. Setting $\Ee_i'$ to be the kernel of the surjection $\Ee_i\twoheadrightarrow \Aa_i$ for each $i=1,2$, we get an exact sequence
\[
0\to \Ee_1' \to \Ee \to \Ee_2' \to 0.
\]
By Remark \ref{propsemi} (b), one gets that $\chi_2=\chi(\Ee_2) \ge \chi(\Ee_2')\ge0$, which implies $d_2 \ge r_2(g_2-1)$. By exchanging the roles of $\Ee_1$ and $\Ee_2$, one also gets $d_1 \ge r_1(g_1-1)$. Thus the decomposition (\ref{decomp}) can be rewritten as 
\begin{equation}\label{decomp1}
\mathbf{U}_C(\mathbf{w}, \vec{r}, \chi) = \bigcup_{0\le \chi_1, \chi_2, \chi_1+\chi_2 \le r_0}\mathbf{U}_C(\mathbf{w}, \vec{r}, \chi)_{(\chi_1, \chi_2)},
\end{equation}
with $r_0:=\mathrm{min}\{r_1, r_2\}$. 



\begin{proposition}\label{propr0}
Let $\Ee$ be a sheaf of depth one on $C$ with $\mathrm{rank}(\Ee) = (r_1, r_2)$ and $r_2 = 0$. Then, $\Ee$ is $\mathbf{w}$-(semi-)stable if and only if $\Ee_1$ in $(\dagger)$ of Remark \ref{rem333} is (semi-)stable. A similar assertion holds when $r_1 = 0$.
\end{proposition}

\begin{proof}
Since $r_2=0$, it follows that $\Ee_2 \cong 0 \cong \Tt$ in $(\dagger)$ and $\Ee_1 \cong \Ee$ in $(\dagger)$. Note that the type of $\Ee$ is given by $\delta(\Ee) = (0, r_1, 0)$. Similarly, for any subsheaf $\Ff \subset \Ee$, we have $\Ff_1 \cong \Ff$. Next, the $\mathbf{w}$-soples of $\Ee$ and $\Ff$ are given by
\[
\mu_{\mathbf{w}}(\Ee) = \frac{\chi(\Ee_1)}{w_1 \cdot r_1} = \left(\frac{1}{w_1}\right)\mu(\Ee_1)+\frac{1-g_1}{w_1}, \quad \text{and} \quad  \mu_{\mathbf{w}}(\Ff) =  \left(\frac{1}{w_1}\right)\mu(\Ff_1)+\frac{1-g_1}{w_1}.
\]
The difference of the slopes can be computed as
\[
\mu_{\mathbf{w}}(\Ee)-\mu_{\mathbf{w}}(\Ff)=\left(\frac{1}{w_1}\right)(\mu(\Ee_1)-\mu(\Ff_1)),
\]
and so we get the assertion.
\end{proof}



\begin{proposition}\label{pprop22}
The moduli space $\mathbf{U}_C(\mathbf{w},1,0)$ is isomorphic to $\Pic^{g_1-1}(C_1)\times \Pic^{g_2-1}(C_2)$ for any polarization $\mathbf{w}$. 
\end{proposition}

\begin{proof}
Let $\Ll$ be a sheaf of depth one in $\mathbf{U}_C(\mathbf{w},1,0)$ with associated exact sequence ($\dagger$) for $\Ll$. When $\tau=0$, we have $\Ll \cong \Ll_1 \oplus \Ll_2$, and since $\chi(\Ll)=0$, it follows that $\chi_1+\chi_2=0$. Now the $\mathbf{w}$-semistability of $\Ll$ implies that each $\chi_i \le 0$, and so $\chi_1=\chi_2=0$, i.e., $d_i = g_i-1$. Conversely, for any line bundle $\Ll_i \in \Pic^{g_i-1}(C_i)$, $\Ll_1 \oplus \Ll_2$ is $\mathbf{w}$-semistable. Thus we get
\[
\mathbf{U}_C(\mathbf{w},1,0)_{(0,0)}\cong \Pic^{g_1-1}(C_1)\times \Pic^{g_2-1}(C_2).
\]

\noindent Now assume that $\tau=1$, i.e., $\delta(\Ll)=(1,0,0)$, and so $\Ll$ fits into the following exact sequence:
\[
0 \to \Ll \to \Ll_1 \oplus \Ll_2 \to \Oo_{p} \to 0.
\]
Since $\chi_1+\chi_2=1$, we may assume that $\chi_1 \ge 1$ without loss of generality. If the composition map $\Ll_1 \hookrightarrow \Ll_1 \oplus \Ll_2 \twoheadrightarrow \Oo_p$ is surjective, then we obtain an exact sequence
\[
0 \to \Ll_1(-p) \to \Ll \to \Ll_2 \to 0.
\]
From the $\mathbf{w}$-semistability of $\Ll$, we get $\chi_1-1=\chi(\Ll_1(-p)) \le 0$, and so $(\chi_1,\chi_2)=(1,0)$. If the composition map is not surjective, then $\Ll_1$ would be a subsheaf of $\Ll$, contradicting the $\mathbf{w}$-semistability of $\Ll$.
Thus, we have 
\[
[\Ll]=[\Ll_1(-p) \oplus \Ll_2] \in \mathbf{U}_C(\mathbf{w},1,0)_{(1,0)}.
\]
This implies, as before, that $\mathbf{U}_C(\mathbf{w},1,0)_{(1,0)} \cong \Pic^{g_1-1}(C_1)\times \Pic^{g_2-1}(C_2)$. From the case when the composition map $\Ll_2 \to \Oo_p$ is surjective, we get $\mathbf{U}_C(\mathbf{w}, 1, 0)_{(0, 1)} \cong \Pic^{g_1 - 1}(C_1) \times \Pic^{g_2 - 1}(C_2)$. Then the assertion follows from (\ref{decomp1}) with $r_0=1$. Note that each point in the moduli space is strictly semistable.
\end{proof}

\begin{proposition}\label{p555}
The moduli space $\mathbf{U}_C(\mathbf{w},(2,1),0)$ is isomorphic to $\mathbf{M}_{C_1}(2,2g_1-2) \times \Pic^{g_2-1}(C_2)$ for any polarization $\mathbf{w}$.
\end{proposition}

\begin{proof}
Let $\Ee$ be a sheaf of depth one in $\mathbf{U}_C(\mathbf{w},(2,1),0)$ with the exact sequence ($\dagger$) with a vector bundle $\Ee_1$ of rank two on $C_1$ and a line bundle $\Ll_2$ on $C_2$.

First assume that $\tau=0$. Then $\Ee$ splits as a direct sum, i.e., $\Ee \cong \Ee_1 \oplus \Ll_2$, and we have $\delta(\Ee) = (0,2,1)$. From the $\mathbf{w}$-semistability of $\Ee$, we get that
\[
\mu_\mathbf{w}(\Ee_1) \le \mu_\mathbf{w}(\Ee)=0, \quad \text{and} \quad \mu_\mathbf{w}(\Ll_2) \le \mu_\mathbf{w}(\Ee)=0
\]
and so $\chi_1=\chi_2=0$. Moreover, one can check that $\Ee_1$ is semistable, and so it corresponds to a point in $\mathbf{M}_{C_1}(2,2g_1-2)$. Conversely, for any $\Ee_1 \in \mathbf{M}_{C_1}(2,2g_1-2)$ and $\Ll_2 \in \Pic^{g_2-1}(C_2)$, $\Ee_1 \oplus \Ll_2$ is $\mathbf{w}$-semistable using Proposition \ref{propr0}. Therefore, we get
\[
\mathbf{U}_C(\mathbf{w},(2,1),0)_{(0,0)} \cong \mathbf{M}_{C_1}(2,2g_1-2) \times \Pic^{g_2-1}(C_2).
\]

Now assume that $\tau=1$. Note that $\Ee$ fits into the following exact sequence:
\[
0 \to \Ee \to \Ee_1 \oplus \Ll_2 \to \Oo_{p} \to 0.
\]
Since $\chi_1+\chi_2=1$, first we assume that $\chi_1 \ge 1$. If the composition map $\Ee_1 \hookrightarrow \Ee_1 \oplus \Ll_2 \twoheadrightarrow \Oo_p$ is surjective, then we obtain an exact sequence
\[
0 \to \Ee_1' \to \Ee \to \Ll_2 \to 0.
\]
From the $\mathbf{w}$-semistability of $\Ee$, we get $\chi_1-1=\chi(\Ee_1') \le 0$, and so $(\chi_1,\chi_2)=(1,0)$. Furthermore, from $\mu_\mathbf{w}(\Ee)=\mu_\mathbf{w}(\Ee_1')=0$ one gets that $\Ee_1'$ is also $\mathbf{w}$-semistable. By Proposition \ref{propr0}, $\Ee_1'$ is semistable as a sheaf on $C_1$, i.e., $[\Ee_1'] \in \mathbf{M}_{C_1}(2,2g_1-2)$. If the composition map is not surjective, then $\Ee_1$ would be a subsheaf of $\Ee$, contradicting the $\mathbf{w}$-semistability of $\Ee$.
Thus, we have 
\[
[\Ee]=[\Ee_1' \oplus \Ll_2] \in \mathbf{U}_C(\mathbf{w},1,0)_{(1,0)}.
\]
Therefore, as in the case $\tau=0$, we get
\[
\mathbf{U}_C(\mathbf{w},(2,1),0)_{(1,0)} \cong \mathbf{M}_{C_1}(2,2g_1-2) \times \Pic^{g_2-1}(C_2).
\]
Similarly, for the case $\chi_2 \ge 1$, we get $[\Ee]=[\Ee_1 \oplus \Ll_2(-p)] \in \mathbf{U}_C(\mathbf{w},1,0)_{(0,1)}$.
In particular, we have 
\[
\mathbf{U}_C(\mathbf{w},(2,1),0)_{(\chi_1,\chi_2)} \cong \mathbf{M}_{C_1}(2,2g_1-2) \times \Pic^{g_2-1}(C_2)
\]
for any $(\chi_1,\chi_2) \in \{(0,0),(1,0),(0,1)\}$.
\noindent Then, we get the assertion by (\ref{decomp1}) with $r_0=1$.
\end{proof}

\begin{proposition}\label{(2,0)(0,0)}
The moduli space $\mathbf{U}_C(\mathbf{w},2,0)_{(0,0)}$ is isomorphic to $\mathbf{M}_{C_1}(2,2g_1-2) \times \mathbf{M}_{C_2}(2,2g_2-2)$ for any $\mathbf{w}$.
\end{proposition}

\begin{proof}
Let $\Ee$ be a sheaf in $\mathbf{U}_C(\mathbf{w},2,0)_{(0,0)}$. Then, $\Ee$ decomposes as $\Ee \cong \Ee_1 \oplus \Ee_2$, where each $\Ee_i$ is a vector bundle of rank $2$ and degree $2g_i-2$ on $C_i$. In particular, each $\Ee_i$ is semistable, i.e., $[\Ee_i] \in \mathbf{M}_{C_i}(2,2g_i-2)$. Define a map
\[
\psi: \mathbf{U}_C(\mathbf{w},2,0)_{(0,0)} \rightarrow \mathbf{M}_{C_1}(2,2g_1-2) \times \mathbf{M}_{C_2}(2,2g_2-2)
\]
by sending $[\Ee]$ to $([\mathfrak{gr}(\Ee_1)], [\mathfrak{gr}(\Ee_2)])$. Clearly, the map $\psi$ is well-defined and injective. Now for the surjectivity, pick a pair $([\Ee_1], [\Ee_2]) \in \mathbf{M}_{C_1}(2,2g_1-2) \times \mathbf{M}_{C_1}(2,2g_2-2)$, and set $\Ee = \Ee_1 \oplus \Ee_2$. For any subsheaf $\Ff$, we have $\Ff \cong \Ff_1 \oplus \Ff_2$ with each $\Ff_i=\Ee_i \cap \Ff$. By Proposition \ref{propr0}, each $\Ee_i$ is $\mathbf{w}$-semistable, and so $\chi(\Ff_i) \le 0$. From the following inequality 
\[
\mu_\mathbf{w}(\Ff) = \frac{\chi(\Ff_1) + \chi(\Ff_2)}{\mathrm{rank}_{\mathbf{w}}(\Ff)} \leq 0 = \mu_\mathbf{w}(\Ee),
\]
we get that $\Ee$ is $\mathbf{w}$-semistable, and the assertion follows.
\end{proof}

\begin{proposition}\label{(2,0)(0,1)}
For any polarization $\mathbf{w}$, we have
\[
\mathbf{U}_C(\mathbf{w},2,0)_{(0,1)} = \mathbf{U}_C(\mathbf{w},2,0)_{(1,0)}=\mathbf{U}_C(\mathbf{w},2,0)_{(0,0)}.
\]
\end{proposition}

\begin{proof}
Without loss of generality, it is enough to show that $\mathbf{U}_C(\mathbf{w},2,0)_{(0,1)} =\mathbf{U}_C(\mathbf{w},2,0)_{(0,0)}$. Let $\Ee$ be a sheaf in $\mathbf{U}_C(\mathbf{w},2,0)_{(0,1)}$ with the following exact sequence ($\dagger$): 
\[
0 \to \Ee \to \Ee_1 \oplus \Ee_2 \to \Oo_{p} \to 0
\]
with a vector bundle $\Ee_i$ of rank two on $C_i$ and $(\chi_1, \chi_2)=(0,1)$. If the composition map $\iota_2 : \Ee_2 \hookrightarrow \Ee_1 \oplus \Ee_2 \twoheadrightarrow \Oo_p$ is not surjective, then $\Ee_2$ would be a subsheaf of $\Ee$, contradicting the $\mathbf{w}$-semistability of $\Ee$. Thus the composition map $\iota_2$ is surjective, and set $\Ee_2' = \ker(\iota_2)$. Then we obtain an exact sequence
\[
0 \to \Ee_2' \to \Ee \to \Ee_1 \to 0.
\]
Moreover, we get $\chi(\Ee_2')=\chi_2-1=0$, and from $\mu_\mathbf{w}(\Ee)=\mu_\mathbf{w}(\Ee_2')=0$ it follows that $\Ee_2'$ is also $\mathbf{w}$-semistable. By Proposition \ref{propr0}, $\Ee_2'$ is semistable as a sheaf on $C_2$, i.e., $[\Ee_2'] \in \mathbf{M}_{C_2}(2,2g_2-2)$. Similarly, we also get $[\Ee_1] \in \mathbf{M}_{C_1}(2,2g_1-2)$. This implies that the Jordan-Hölder filtration of $\Ee$ is given by
\[
\mathfrak{gr}_\mathbf{w}(\Ee) \cong \mathfrak{gr}(\Ee_1) \oplus \mathfrak{gr}(\Ee_2').
\]
Thus, by Proposition \ref{(2,0)(0,0)}, every sheaf in $\mathbf{U}_C(\mathbf{w},2,0)_{(0,1)}$ is contained in $\mathbf{U}_C(\mathbf{w},2,0)_{(0,0)}$.

Conversely, let $[\Ee_1' \oplus \Ee_2'] \in \mathbf{U}_C(\mathbf{w}, 2, 0)_{(0,0)}$ with $[\Ee_i']\in \mathbf{M}_{C_i}(2,2g_i-2)$ by Proposition \ref{(2,0)(0,0)}. Let $\Ee'$ be a nontrivial extension 
\begin{equation}\label{eerrr1}
0 \to \Ee_2' \to \Ee' \to \Oo_p \to 0
\end{equation}
of $\chi(\Ee')=1$, which is represented by a general element in $\PP \mathrm{Ext}_{C_2}^1(\Oo_p, \Ee_2')\cong  \PP^1$. Indeed, if $\Ee_2'\in \mathbf{M}_{C_2}(2,2g_2-2)$ is stable, then $\Ee'$ is also stable for any element in the extension family. Assume that $\Ee_2'$ is strictly semistable, say $[\Ee_2']=[\Ll_2\oplus \Ll_2']$ with $\Ll_2, \Ll_2' \in \Pic^{g_2-1}(C_2)$. Then $\Ee'$ is expressed in terms of coordinates
\[
[\epsilon_0:\epsilon_1] \in \PP \mathrm{Ext}_{C_2}^1(\Oo_p, \Ll_2\oplus \Ll_2')\cong \PP \left(\mathrm{H}^0(\Ll_2^\vee \otimes \Oo_p)^\vee \oplus \mathrm{H}^0(\Ll_2'^\vee \otimes \Oo_p)^\vee \right)\cong \PP^1.
\]
If $\Ll\subset \Ee'$ is a destabilizing line bundle with $\deg (\Ll)\ge g_2$, then the composition map $\upsilon: \Ll\rightarrow \Ee' \rightarrow \Oo_p$ is surjective; otherwise, $\Ll$ would be a subsheaf of $\Ll_2\oplus \Ll_2'$, impossible. Then $\Ll(-p)\cong \mathrm{ker}(\upsilon)$ is a subsheaf of $\Ll_2\oplus \Ll_2'$ and so we get $\Ll\cong \Ll_2(p)$ or $\Ll_2'(p)$. In the former case, we get that $\Ee'$ is an extension of $\Ll_2'$ by $\Ll_2(p)$, which is trivial. Similarly, in the latter case, we get $\Ee' \cong \Ll_2\oplus \Ll_2'(p)$. In particular, the coordinates $[\epsilon_0, \epsilon_1] \not\in \{ [0:1], [1:0]\}$ if and only if $\Ee'$ is semistable, i.e., $[\Ee']\in \mathbf{M}_{C_2}(2,2g_2-1)$. For any such $\Ee'$, set $\Ee:=\Ee_1' \oplus \Ee'$. Then $\Ee\in \mathbf{U}_C(\mathbf{w}, 2,0)_{(0,1)}$ maps to $[\Ee_1'\oplus \Ee_2']$, and the assertion follows. 
\end{proof}


\begin{remark}\label{rrmml}
From the argument in the proof of Proposition \ref{(2,0)(0,1)}, one can obtain a rational map 
\[
\PP \mathrm{Ext}_{C_2}^1(\Oo_p, \Ee_2') \cong \PP^1 \dashrightarrow \mathbf{M}_{C_2}(2,2g_2-1),
\]
which must be a morphism. If $C_2$ is an elliptic curve and so $\mathbf{M}_{C_2}(2,1)\cong C_2$, then by the Hurwitz formula the morphism must be a constant map. In particular, the vector bundle $\Ee'$ in (\ref{eerrr1}) is unique, which is not true for $g_2\ge 2$.
\end{remark}

\begin{proposition}\label{(2,0)(0,2)}
For any polarization $\mathbf{w}$, we have
\[
\mathbf{U}_C(\mathbf{w},2,0)_{(0,2)} = \mathbf{U}_C(\mathbf{w},2,0)_{(2,0)}=\mathbf{U}_C(\mathbf{w},2,0)_{(0,0)}.
\]
\end{proposition}

\begin{proof}
Without loss of generality, it is enough to show that $\mathbf{U}_C(\mathbf{w},2,0)_{(0,2)} =\mathbf{U}_C(\mathbf{w},2,0)_{(0,0)}$. Let $\Ee$ be a sheaf in $\mathbf{U}_C(\mathbf{w},2,0)_{(0,2)}$ with the exact sequence ($\dagger$)
\[
0 \to \Ee \to \Ee_1 \oplus \Ee_2 \to \Oo_{p}^{\oplus 2} \to 0,
\]
 with a vector bundle $\Ee_i$ of rank two on $C_i$ and $(\chi_1, \chi_2)=(0,2)$. If the composition map $\iota_2 : \Ee_2 \hookrightarrow \Ee_1 \oplus \Ee_2 \twoheadrightarrow \Oo_p^{\oplus 2}$ is not a surjection, then $\ker (\iota_2)$ is a subsheaf of $\Ee$ with Euler characteristic at least $1$, contradicting the $\mathbf{w}$-semistability of $\Ee$. Thus $\iota_2$ is surjective, and we get an exact sequence
\[
0 \to \Ee_2':=\ker (\iota_2) \to \Ee \to \Ee_1 \to 0,
\]
where $\Ee_2'$ has degree $2g_2-2$. Note from $\mu_\mathbf{w}(\Ee)=\mu_\mathbf{w}(\Ee_2')=0$ that $\Ee_2'$ is also $\mathbf{w}$-semistable. By Proposition \ref{propr0}, $\Ee_2'$ is semistable as a sheaf on $C_2$, i.e., $[\Ee_2'] \in \mathbf{M}_{C_2}(2,2g_2-2)$. Similarly, using Remark \ref{propsemi} (b), we also get $[\Ee_1] \in \mathbf{M}_{C_1}(2,2g_1-2)$. This implies that the Jordan-Hölder filtration of $\Ee$ is given by
\[
\mathfrak{gr}_\mathbf{w}(\Ee) \cong \mathfrak{gr}(\Ee_1) \oplus \mathfrak{gr}(\Ee_2').
\]
Thus, by Proposition \ref{(2,0)(0,0)}, every sheaf in $\mathbf{U}_C(\mathbf{w},2,0)_{(0,2)}$ is contained in $\mathbf{U}_C(\mathbf{w},2,0)_{(0,0)}$.

Conversely, let $[\Ee_1'\oplus \Ee_2'] \in \mathbf{U}_C(\mathbf{w},2,0)_{(0,0)}$ for $[\Ee_i'] \in \mathbf{M}_{C_i}(2,2g_i-2)$ by Proposition \ref{(2,0)(0,0)}. For $\Ee_2:= \Ee_2'(p) \in \mathbf{M}_{C_2}(2,2g_2)$, consider a surjection $\sigma_2:\Ee_2 \rightarrow \Oo_p^{\oplus 2}$ so that we have an exact sequence
\[
0 \to \Ee_2' \to \Ee_2 \to \Oo_p^{\oplus2} \to 0.
\]
Then, for the map $(0,\sigma_2):\Ee_1' \oplus \Ee_2 \rightarrow \Oo_p^{\oplus2}$, we get $\ker(0,\sigma_2)=\Ee_1' \oplus \Ee_2'$ and the assertion follows.
\end{proof}

\begin{proposition}\label{(2,0)(1,1)}
The locus $\mathbf{U}_C(\mathbf{w},2,0)_{(0,0)}$ is contained in $\mathbf{U}_C(\mathbf{w}, 2, 0)_{(1,1)}$. In particular, we have
\[
\mathbf{U}_C(\mathbf{w}, 2, 0)=\mathbf{U}_C(\mathbf{w}, 2, 0)_{(1,1)}.
\]
\end{proposition}

\begin{proof}
Let $[\Ee_1'\oplus \Ee_2'] \in \mathbf{U}_C(\mathbf{w},2,0)_{(0,0)}$ for $[\Ee_i'] \in \mathbf{M}_{C_i}(2,2g_i-2)$ by Proposition \ref{(2,0)(0,0)}. Consider an elementary transformation
\[
0\to \Ee_i' \to \Ee_i \xrightarrow{\phi_i} \Oo_p \to 0
\]
with $\Ee_i\in \mathbf{M}_{C_i}(2,2g_i-1)$ as in the proof of Proposition \ref{(2,0)(0,1)}. Then we get an exact sequence
\[
0 \to \Ee_1'\oplus \Ee_2' \to \Ee':=\Ee_1 \oplus \Ee_2 \xrightarrow{\phi} \Oo_p^{\oplus2} \to 0
\]
with $\phi=\begin{pmatrix} \phi_1 & 0 \\ 0 & \phi_2 \end{pmatrix}$. Since the sheaf $\Ee'$ has $\mathrm{rank}(\Ee')=(2,2)$ and $\chi_1=\chi_2=1$ by construction, we get the first assertion. Now the second assertion follows from Propositions \ref{(2,0)(0,1)} and \ref{(2,0)(0,2)}. 
\end{proof}

\begin{lemma}\label{minus}
Each sheaf $\Ee \in \mathbf{U}_C(\mathbf{w}, 2, 0)\setminus \mathbf{U}_C(\mathbf{w}, 2, 0)_{(0,0)}$ is $\mathbf{w}$-stable for any polarization. Moreover, in the corresponding sequence $(\dagger)$ for $\Ee$, the vector bundle $\Ee_i$ is stable for each $i$, i.e., $\Ee_i\in \mathbf{M}_{C_i}(2,2g_i-1)$. 
\end{lemma}

\begin{proof}
Let $\Ee$ be a strictly $\mathbf{w}$-semistable sheaf in $\mathbf{U}_C(\mathbf{w}, 2, 0)$ with $\Ff$ the maximal subsheaf of $\Ee$ in the Jordan-Hölder filtration of $\Ee$. In particular, the quotient $\Ee/\Ff$ is $\mathbf{w}$-stable with $\mu_{\mathbf{w}}(\Ee/\Ff)=\mu_{\mathbf{w}}(\Ee)=0$. By Propositions \ref{pprop22} and \ref{p555}, we get that $\mathrm{rank}(\Ee/\Ff)$ is either $(r,0)$ or $(0,r)$ for some $r\in \{1,2\}$, say $(r,0)$. If $r=1$, then we have $\Ee/\Ff \cong \Ll_1$ for some $\Ll_1\in \mathrm{Pic}^{g_1-1}(C_1)$. Then by Proposition \ref{p555} we get $[\Ff] = [\Ll_1'\oplus\Ee_2']$ for $\Ll_1'\in \mathrm{Pic}^{g_1-1}(C_1)$ and $\Ee_2' \in \mathbf{M}_{C_2}(2,2g_2-2)$. Similarly, if $r=2$, then we have $\Ee/\Ff \cong \Ee_1'$ for some $\Ee_1'\in \mathbf{M}_{C_1}(2,2g_1-2)$. Then by Proposition \ref{pprop22} we get $[\Ff] \in \mathbf{M}_{C_2}(2,2g_2-2)$. Thus $[\Ee]\in \mathbf{U}_C(\mathbf{w}, 2, 0)_{(0,0)}$, and the first assertion follows. 

Now let $\Ee$ be a sheaf in $\Ee \in \mathbf{U}_C(\mathbf{w}, 2, 0)\setminus \mathbf{U}_C(\mathbf{w}, 2, 0)_{(0,0)}$, i.e., $\Ee$ is $\mathbf{w}$-stable and fits into the following exact sequence ($\dagger$):
\[
0 \to \Ee \to \Ee_1 \oplus \Ee_2 \to \Oo_{p}^{\oplus 2} \to 0,
\]
where each $\Ee_i$ is a vector bundle of rank two on $C_i$ with $\chi_i=1$ by Proposition \ref{(2,0)(1,1)}. If $\Ee_i$ is not stable, then there exists a line subbundle $\Ll_i$ with $\deg(\Ll_i) \ge g_i$. Then the kernel of the composition map $\Ll_i \hookrightarrow \Ee_1 \oplus \Ee_2 \twoheadrightarrow \Oo_p^{\oplus2}$ has Euler characteristic at least $0$, and it injects into $\Ee$ as a subsheaf, violating the $\mathbf{w}$-stability of $\Ee$. Thus we can conclude that both $\Ee_1$ and $\Ee_2$ are stable. 
\end{proof}

\begin{lemma}\label{ker(0,0)}
For vector bundles $\Ee_i'$ with $[\Ee_i'] \in \mathbf{M}_{C_i}(2,2g_i-2)$, there exist stable vector bundles $\Ee_i \in \mathbf{M}_{C_i}(2,2g_i-1)$ and a surjection $\sigma = (\sigma_1, \sigma_2) : \Ee_1 \oplus \Ee_2 \to \Oo_p^{\oplus 2}$ satisfying
\[
[\mathrm{ker}(\sigma)] = [\Ee_1' \oplus \Ee_2'] \in \mathbf{U}_C(\mathbf{w}, 2, 0)_{(0,0)},
\]
such that $\sigma_2: \Ee_2 \to \Oo_p^{\oplus 2}$ is also surjective.
\end{lemma}

\begin{proof}
As in the proof of Proposition \ref{(2,0)(0,1)}, for each $i=1,2$, there exists a stable vector bundle $\Ee_i\in \mathbf{M}_{C_i}(2,2g_i-1)$ such that $\Ee_i'$ fits into an exact sequence 
\begin{equation}\label{eeerrr1}
0 \to \Ee_i' \to \Ee_i \xrightarrow{\phi_i} \Oo_p \to 0.
\end{equation}

Fix an identification $\eta_i: \Ee_{i, q_i} \stackrel{\simeq}{\longrightarrow} \Oo_p^{\oplus 2}$ and set $\mathrm{res}_i: \Ee_i \rightarrow \Ee_{i, q_i}$ to be the restriction map. Then one can find a map $\xi_1: \Oo_p^{\oplus 2} \rightarrow \Oo_p$ such that $\phi_1=\xi_1\circ \eta_1 \circ \mathrm{res}_1$. Set $\sigma_1=\varrho \circ \phi_1$ and $\sigma_2=\eta_2 \circ \mathrm{res}_2$, where $\varrho: \Oo_p \hookrightarrow \Oo_p^{\oplus2}$ is an injective map. Then, we can find the map $\varrho$, for which the induced map $\phi_2': \Ee_2 \rightarrow \Oo_p$ in the following commutative diagram is identical to $\phi_2$.
\begin{equation}\label{stabdiag2}
\begin{tikzcd}[
row sep=normal, column sep=normal,
  ar symbol/.style = {draw=none,"\textstyle#1" description,sloped},
  isomorphic/.style = {ar symbol={\cong}},
  ]
  &   0\ar[d]     & 0 \ar[d] & 0\ar[d]\\
  0 \ar[r] & \Ee_1' \ar[d]\ar[r] & \Ee_1 \ar[d]\ar[r,"\phi_1"] &\Oo_p \ar[d, "\varrho"] \ar[r] & 0  \\
  0 \ar[r] & \mathrm{ker}(\sigma)\ar[d] \ar[r]             & \Ee_1\oplus \Ee_2  \ar[d]\ar[r, "\sigma"] & \Oo_p^{\oplus 2} \ar[d] \ar[r] & 0\\
      0    \ar[r]   &  \mathrm{ker}(\phi_2')     \ar[r] \ar[d]      & \Ee_2 \ar[d]\ar[r, "\phi_2'"
      ] & \Oo_p \ar[d]  \ar[r]&0 \\
    & 0 & 0 & 0
\end{tikzcd}
\end{equation}
In particular, we get $\ker(\phi_2') = \Ee_2'$ and so $[\ker(\sigma)] = [\Ee_1' \oplus \Ee_2']$. Thus the assertion follows.
\end{proof}

Since $C_i$ is projective, the projection $\pi_{i1}: \mathbf{M}_{C_i}(2,2g_i-1) \times C_i \rightarrow \mathbf{M}_{C_i}(2,2g_i-1)$ is a proper morphism. The Poincar\'e universal bundle $\Pp_i$ is flat over $\mathbf{M}_{C_i}(2,2g_i-1)$ by universality, and so $\Pp_i \otimes \pi_{12}^* \Oo_{q_i}$ is flat over $\mathbf{M}_{C_i}(2,2g_i-1)$, where $\pi_{i2}: \mathbf{M}_{C_i}(2,2g_i-1) \times C_i \rightarrow C_i$. Moreover, the fibres of $\pi_{i1}$ are the smooth curves $C_i$, and the sheaf $\Pp_i \otimes \pi_{i2}^* \Oo_{q_i}$ has constant rank $2$ on these fibres. By the Grauert theorem in \cite{hartshorne}, since 
\[
\mathrm{h}^0\left(\{\Ee_i\}\times C_i, \left(\Pp_i \otimes \pi_{i2}^* \Oo_{q_i}\right)_{|\{\Ee_i\}\times C_i}\right)=2
\]
for each $\Ee_i \in \mathbf{M}_{C_i}(2,2g_i-1)$, the direct image sheaf $\pi_{i1*} (\Pp_i \otimes \pi_{i2}^* \Oo_{q_i})$ is a vector bundle of rank $2$ over $\mathbf{M}_{C_i}(2,2g_i-1)$. Setting $\pi_1$ and $\pi_2$ to be the two natural projections from $\mathbf{M}_{C_1}(2,2g_1-1)\times \mathbf{M}_{C_2}(2,2g_2-1)$, one can define a vector bundle $\EE$ on $\mathbf{M}_{C_1}(2,2g_1-1)\times \mathbf{M}_{C_2}(2,2g_2-1)$:
\begin{equation}
\EE:= \pi_1^* \pi_{11*} \left ( \Pp_1 \otimes \pi_{12}^* \Oo_{q_1}\right)^\vee \otimes \pi_2^* \pi_{21*} \left ( \Pp_2 \otimes \pi_{22}^* \Oo_{q_2}\right).
\end{equation}
Consequently, $\EE$ is a vector bundle of rank 4, and the fibre of $\EE$ over $(\Ee_1, \Ee_2)$ is identified with
\[
\mathrm{Hom}\left(\Ee_{1,q_1}, \Ee_{2,q_2}\right).
\]
Inside the projective bundle $\PP(\EE)$, one can consider the locus $\mathbf{Q}$:
\[
\mathbf{Q}:=\left\{(\Ee_1, \Ee_2, f)\in \PP(\EE)~|~ \mathrm{rank}(f)=1\right\}.
\]
Over each point $(\Ee_1, \Ee_2)\in \mathbf{M}_{C_1}(2,2g_1-1)\times \mathbf{M}_{C_2}(2,2g_2-1)$, the fibre of $\mathbf{Q}$ is defined as the zero locus of the determinant of the homomorphisms in $\PP \mathrm{Hom}\left(\Ee_{1,q_1}, \Ee_{2,q_2}\right)\cong \PP^3 $, and so is isomorphic to $\PP^1\times \PP^1$, a smooth quadric surface. Let $(\Ee_1, \Ee_2, f) \in \PP(\EE)$. For fixed isomorphisms $\eta_i : \Ee_{i,q_i} \xrightarrow{\simeq} \Oo_p^{\oplus 2}$, one can define a homomorphism $\bar{f}:= \eta_2 \circ f \circ {\eta_1}^{-1}: \Oo_p^{\oplus 2} \rightarrow \Oo_p^{\oplus 2}$. Using this, we define the maps $\sigma_1$ and $\sigma_2$ as follows:
\[
\sigma_1 := \bar{f} \circ \eta_1 \circ \mathrm{res}_1: \Ee_1 \to \Oo_p^{\oplus 2}, \quad \sigma_2 := \eta_2 \circ \mathrm{res}_2 : \Ee_2 \to \Oo_p^{\oplus 2}.
\]
Then the map $(\sigma_1, \sigma_2)$ is surjective onto $\Oo_p^{\oplus2}$, and set
\[\
\Ee_{(\sigma_1, \sigma_2)} := \ker\left(\Ee_1 \oplus \Ee_2 \xrightarrow{(\sigma_1,\sigma_2)} \Oo_p^{\oplus 2}\right).
\]

\begin{proposition}\label{sub}
For each $(\Ee_1, \Ee_2, f) \in \PP(\EE)$, the sheaf $\Ee_{(\sigma_1, \sigma_2)}$ is $\mathbf{w}$-semistable, and so one can define a morphism 
\[
\Phi : \PP (\EE) \rightarrow \mathbf{U}_C(\mathbf{w}, 2, 0)
\]
such that $\Phi(\PP(\EE)\setminus \mathbf{Q})\subseteq \mathbf{U}_C(\mathbf{w}, 2, 0)\setminus \mathbf{U}_C(\mathbf{w}, 2, 0)_{(0,0)}$ and $\Phi(\mathbf{Q})\subseteq \mathbf{U}_C(\mathbf{w}, 2, 0)_{(0,0)}$.
\end{proposition}

\begin{proof}
For the $\mathbf{w}$-semistability of $\Ee_{(\sigma,_1, \sigma_2)}$, we divide into two cases depending on the rank of $f$. 

\quad (a) First we assume that $f: \Ee_{1,q_1}\rightarrow\Ee_{2,q_2}$ has rank one, and so we get $\mathrm{Im}(\sigma_1) \cong \Oo_p$. Let ${\bar{\sigma}}_1:\Ee_1 \rightarrow \Oo_p$ be the composition of $\sigma_1$ with projection onto its image $\mathrm{im}(\sigma_1) \cong \Oo_p$. Then we get an exact sequence
\[
0 \rightarrow \ker({\bar{\sigma}}_1) \rightarrow \Ee_1 \xrightarrow{\ {\bar{\sigma}}_1} \Oo_p \rightarrow 0,
\]
from which we get $[\ker({\bar{\sigma}}_1)] \in \mathbf{M}_{C_1}(2,2g_1-2)$. Since the composition $\mathrm{ker}({\bar{\sigma}}_1) \hookrightarrow \Ee_1\oplus \Ee_2 \rightarrow \Oo_p^{\oplus 2}$ is a zero map, we get an injection $\ker({\bar{\sigma}}_1) \hookrightarrow \Ee_{(\sigma_1, \sigma_2)}$, whose cokernel is isomorphic to the kernel of the map $\bar{\sigma}_2: \Ee_2 \rightarrow \Oo_p$, induced from $\sigma_2$. Since $\ker(\bar{\sigma}_2)$ represents a class in $\mathbf{M}_{C_2}(2,2g_2-2)$, we have
\[
\mathfrak{gr}_\mathbf{w}(\Ee_{(\sigma_1, \sigma_2)}) = \mathfrak{gr}(\ker({\bar{\sigma}}_1)) \oplus \mathfrak{gr}(\ker(\bar{\sigma}_2)).
\]
Therefore, by Proposition \ref{(2,0)(0,0)}, we get $\Ee_{(\sigma_1, \sigma_2)}\in\mathbf{U}_C(\mathbf{w}, 2, 0)_{(0,0)}$.

\quad (b) Now we consider the case when $f$ has rank two. Suppose for contradiction that $\Ee_{(\sigma_1, \sigma_2)}$ is not $\mathbf{w}$-semistable with a destabilizing torsion-free subsheaf $\Ff \subset \Ee_{(\sigma_1, \sigma_2)}$, i.e., $\mu_{\mathbf{w}}(\Ff)>0$. By Lemma \ref{commute}, we have an exact sequence
\[
0 \to \Ff \to \Ff_1 \oplus \Ff_2 \to \Tt' \to 0,
\]
where $\Ff_i := \Ff_{|C_i}$ and $\Tt'$ is a torsion subsheaf of $\Tt$, supported at $p$. From the stability of $\Ee_i$, we get $\chi(\Ff_i) \leq 0$ for each $i$. It implies that 
\[
\chi(\Ff) \leq \chi(\Ff_1) + \chi(\Ff_2)\le 0,
\]
a contradiction, and so we get $[\Ee_{(\sigma_1, \sigma_2)}] \in \mathbf{U}_C(\mathbf{w}, 2, 0)$. Now assume further that $\Ee_{(\sigma_1, \sigma_2)}$ is strictly $\mathbf{w}$-semistable. Then one can find a $\mathbf{w}$-stable subsheaf $\Ff \subsetneq \Ee_{(\sigma_1, \sigma_2)}$ with $\mu_{\mathbf{w}}(\Ff)=0$. From a Jordan–Hölder filtration of $\Ee_{(\sigma_1, \sigma_2)}$. By Propositions \ref{pprop22} and \ref{p555}, we get $\mathrm{Supp}(\Ff)=C_i$ for some $i$. By Lemma \ref{commute}, we get an injection $\Ff \into \Ee_i$ so that the composition $\Ff \hookrightarrow \Ee_1 \oplus \Ee_2 \twoheadrightarrow \Oo_p^{\oplus2}$ is a zero map. In particular, the map $\Ee_1/\Ff \rightarrow \Oo_p^{\oplus2}$ is surjective, absurd. Therefore, $\Ee_{(\sigma,_1, \sigma_2)}$ must be $\mathbf{w}$-stable, and we get $\Phi(\PP(\EE)\setminus \mathbf{Q})\subseteq \mathbf{U}_C(\mathbf{w}, 2, 0)\setminus \mathbf{U}_C(\mathbf{w}, 2, 0)_{(0,0)}$. 

The inclusion $\Phi(\mathbf{Q})\subseteq \mathbf{U}_C(\mathbf{w}, 2, 0)_{(0,0)}$ follows from Lemma \ref{ker(0,0)}.
\end{proof}

\begin{proof}[Proof of Theorem \ref{main}]
First, let us define a rational map $\Psi: \mathbf{U}_C(\mathbf{w}, 2, 0) \dashrightarrow \PP(\EE)$. Let $\Ee$ be a $\mathbf{w}$-stable sheaf in $\mathbf{U}_C(\mathbf{w}, 2, 0)$. Then by Lemma \ref{minus}, it admits an exact sequence $(\dagger)$ with each $\Ee_i\in \mathbf{M}_{C_i}(2,2g_i-1)$. If the map ${{\hat\eta}_i}:= \Ee_{i,q_i} \rightarrow \Oo_p^{\oplus 2}$ is not an isomorphism for some $i$, then $\ker(\phi_i)$ would violate the $\mathbf{w}$-stability of $\Ee$ as in Lemma \ref{minus}. Thus, the map $f= ({{\hat\eta}_2})^{-1} \circ {\hat\eta}_1: \Ee_{1,q_1} \rightarrow \Ee_{2,q_2}$ has rank $2$, and we get $(\Ee_1, \Ee_2, f) \in \PP(\EE) \setminus \mathbf{Q}$. Clearly, $\Psi$ is a birational inverse to $\Phi$ so that $\Phi$ is an isomorphism between $\PP(\EE) \setminus \mathbf{Q}$ and $\mathbf{U}_C(\mathbf{w}, 2, 0)\setminus \mathbf{U}_C(\mathbf{w}, 2, 0)_{(0,0)}$. 

Assume that $g=2$, i.e., $(g_1, g_2)=(1,1)$. For the class of a strictly $\mathbf{w}$-semistable sheaf $[\Ll_1\oplus \Ll_1'\oplus \Ll_2\oplus \Ll_2']\in \mathbf{U}_C(\mathbf{w}, 2, 0)_{(0,0)}$, we can have a pair of stable vector bundles $(\Ee_1, \Ee_2)$ as in Lemma \ref{ker(0,0)}. Then the map $f:=\eta_2^{-1}\circ \varrho \circ {\bar\phi}_1 : \Ee_{1,q_1} \rightarrow \Ee_{2,q_2}$ has rank one, where ${{\bar\phi}_1}: \Ee_{1,q_1} \rightarrow \Oo_p$ is the map induced by $\phi_1$, and so we get $(\Ee_1, \Ee_2, f)\in \mathbf{Q}\subset \PP(\EE)$. Thus $\Psi$ is an inverse morphism on $\mathbf{U}_C(\mathbf{w}, 2, 0)$. 

Lastly, assume that $g>2$ and so $g_i\ge 2$ for some $i$, say $i=1$. In this case, we have $\dim\mathbf{Q}=4g-4$, while we have
\[
\dim \mathbf{U}_C(\mathbf{w}, 2, 0)_{(0,0)}=
\begin{cases}
4g-6, & \text{if} \quad g_2 \ge 2 \\
4g-5, & \text{if} \quad g_2=1.
\end{cases}
\] 
In particular, the map $\Psi$ cannot extend to a morphism.
\end{proof}

\begin{remark}\label{QQ}
Assume that each $C_i$ is an elliptic curve, i.e., $g=2$. From the proof of Theorem \ref{main}, one gets 
\[
\mathbf{Q} \cong \mathbf{U}_C(\mathbf{w}, 2, 0)_{(0,0)} \cong \mathrm{Sym}^2(C_1)\times \mathrm{Sym}^2(C_2),
\]
which admits the structure of a $(\PP^1\times \PP^1)$-fibration over $\mathbf{M}_{C_1}(2,1)\times \mathbf{M}_{C_2}(2,1)\cong C_1\times C_2$. For a fixed pair $(\Ee_1,\Ee_2) \in \mathbf{M}_{C_1}(2,1)\times \mathbf{M}_{C_2}(2,1)$, each surjection $\phi_i \in \Hom(\Ee_i,\Oo_p)$ defines a class $[\Ll_i \oplus \Ll_i'] \in \mathbf{M}_{C_i}(2,0)$ uniquely up to scalar by Proposition \ref{(2,0)(0,1)}. Thus, the fibre in $\mathbf{Q}$ over $(\Ee_1,\Ee_2)$ is isomorphic to
\[
\PP\left(\Hom(\Ee_1,\Oo_p)\right) \times \PP\left(\Hom(\Ee_2,\Oo_p)\right)
 \cong \PP^1\times\PP^1.
\]
Indeed, the fibration map $\mathrm{Sym}^2(C_1) \times \mathrm{Sym}^2(C_2) \rightarrow C_1 \times C_2$ can be defined as follows. Consider the following composition of morphisms:
\[
\mathrm{Sym}^2(C_i) \xrightarrow{\rm{Remark} \ref{genus} (a\text{-}1)} \mathbf{M}_{C_i}(2,0) \xrightarrow{\rm{Remark} \ref{rrmml}} \mathbf{M}_{C_i}(2,1) \xrightarrow{\rm{Remark} \ref{genus} (a\text{-}2)} C_i.
\]
Under the identification $\mathrm{Sym}^2(C_i) \cong \mathbf{M}_{C_i}(2,0)$, a set $\{t_i, t_i'\}$ corresponds to the $S$-equivalence class of the vector bundle $\Oo_{C_i}(t_i - q_i) \oplus \Oo_{C_i}(t_i' - q_i)$. By Remark~\ref{rrmml}, this class maps to the unique stable vector bundle in $\mathbf{M}_{C_i}(2,1)$, which is a nontrivial extension $\Ee_i$ of $\Oo_{q_i}$ by $\Oo_{C_i}(t_i - q_i) \oplus \Oo_{C_i}(t_i' - q_i)$. Then, one can check that $\Ee_i$ is also the nontrivial extension of $\Oo_{C_i}(t_i + t_i' - q_i)$ by $\Oo_{C_i}$. Since $\Oo_{C_i}(t_i + t_i' - q_i) \cong \Oo_{C_i}(\bar{t}_i)$ for a unique point $\bar{t}_i \in C_i$, the bundle $\Ee_i$ maps to the point $\bar{t}_i$ under the final identification $\mathbf{M}_{C_i}(2,1) \cong C_i$. From the group law $\oplus$ on $C_i$ given by
\[
\Oo_{C_i}(t_i - q_i) \otimes \Oo_{C_i}(t_i' - q_i) = \Oo_{C_i}(t_i + t_i' - 2q_i).
\]
with $q_i$ as its identity element, we get $\bar{t}_i=t_i \oplus t_i'$. Overall, the fibration map 
\[
\mathrm{Sym}^2(C_1) \times \mathrm{Sym}^2(C_2) \rightarrow C_1 \times C_2
\]
sends a pair $\big(\{t_1,t_1'\},\{t_2,t_2'\}\big)$ to $({t_1 \oplus t_1'},{t_2 \oplus t_2'}) $.
\end{remark}
\section{Virtual Poincar\'e polynomial of the Moduli space}\label{sec4}
In this section, we compute the virtual Poincar\'e polynomial of the moduli space $\mathbf{U}_C(\mathbf{w}, 2, 0)$ explicitly, using a concrete description of the moduli space. The argument is based on Kirwan’s desingularization together with the properties of the virtual Poincar\'e polynomial.

The virtual Poincar\'e polynomial $\mathrm{P}(X)$ of a quasi-projective variety $X$ over the field $\CC$ of complex numbers is defined with the motivic property by P.~Deligne. If $X$ is smooth and projective, then it is identical to the classical Poincar\'e polynomial $\displaystyle\sum_i (-1)^i \dim \mathrm{H}^i(X) t^i$. Recall that $\mathrm{P}(\PP^n)=\displaystyle\sum_{i=0}^n t^{2i}$, $\mathrm{P}(D)=1+2gt+t^2$ for a smooth projective curve $D$ of genus $g$, and $\mathrm{P}(A)=(1+t)^{2g}$ for an abelian variety $A$ of dimension $g$. We also have
\begin{equation}\label{cg}
C_g(t):=\mathrm{P}(\mathbf{M}_{D}(2,d))=(1+t)^{2g}\left(\frac{(1+t^3)^{2g}-t^{2g}(1+t)^{2g}}{(1-t^2)(1-t^4)}\right)
\end{equation}
for $d$ coprime to $2$, by \cite{newstead}. 

Now, consider the case when $d$ is even. Since $\mathbf{M}_{D}(2,d)\cong\mathbf{M}_{D}(2,d+2e)$ for any $e\in\ZZ$, it suffices to consider $\mathbf{M}_{D}(2,0)$. It is known from \cite{Ful93} that for any orbifold $X$, the virtual Poincar\'{e} polynomial $\mathrm{P}(X)$ coincides with the Poincar\'{e} polynomial for the ordinary cohomology with compactly supported $\mathrm{H}_{c}^{*}(X)$. Let $\mathbf{M}_{2}$ be the Kirwan's partial desingularization of $\mathbf{M}_{D}(2,0)$. By \cite[Proposition 6.9]{Kir85}, $\mathbf{M}_{2}$ is obtained by the composition of the first blow-up $\mathbf{M}_{1}$ of $\mathbf{M}_{D}(2,0)$ along $J=\Pic^{0}(D)$ and the second blow-up of $\mathbf{M}_{1}$ along the proper transform of $(J\times J)/\ZZ_{2}$:
\[
\mathbf{M}_2 \to \mathbf{M}_1 \to \mathbf{M}_D(2,0).
\]
Let $D_{1}$ be the proper transform of the exceptional divisor of the first blow-up and $D_{2}$ be the exceptional divisor of the second blow-up. $\mathbf{M}_{2}$ has at worst orbifold singularities. If we blow up $\mathbf{M}_{2}$ along the orbifold singularities, we get a smooth variety $\mathbf{K}$, the Kirwan's desingularization. Let $\tilde{D}_{1}$ be the proper transform of $D_{1}$ in $\mathbf{K}$.

\begin{proposition}[Proposition 4.9 of \cite{K}]\label{PP of M_2}
The Poincar\'{e} polynomial of $\mathbf{M}_{2}$ is given by 
\begin{align*}
\mathrm{P}(\mathbf{M}_{2})=&\frac{(1+t)^{2g}}{1-t^{2}}\left\{((1+t^{3})^{2g}-t^{2g+2}(1+t)^{2g})\frac{1}{1-t^{4}}\right.+\frac{t^{2}(1-t^{6g-2})}{1-t^{4}}-t^{4g-2}\frac{1-t^{2g}}{1-t^{2}}\\
&+\frac{t^{2}(1-t^{4g-6})}{1-t^{4}}\left(\frac{1}{2}(1+t)^{2g}(1+t^{2})+\frac{1}{2}(1-t)^{2g}(1-t^{2})+\frac{t^{2}(1-t^{2g-2})}{1-t^{2}}\right)\\
&\left.-t^{2g-2}\frac{1-t^{2g-2}}{1-t^{2}}\left((1+t)^{2g}+\frac{t^{2}(1-t^{2g-2})}{1-t^{2}}\right)\right\}.
\end{align*}
\end{proposition}

Let $\Ss_j$ be the tautological bundle of rank $j$ over the Grassmannian $\mathrm{Gr}(j,g)$ for $j=2,3$. For $B\in\mathrm{Gr}(3,g)$, the variety of complete conics $\mathbf{C}(B)$ is defined as the closure of the locus of pairs $(D, D')\in \PP(\mathrm{Sym}^{2}B)\times \PP(\mathrm{Sym}^{2}B^{\vee})$ with $D$ and $D'$ smooth conics that are dual to each other. Then it is realized as the blow-up
\[
\xymatrix{\PP(\mathrm{Sym}^{2}B)&\mathbf{C}(B)\ar[l]_{\hspace{.5cm}\Phi_{B}}\ar[r]^{\Phi_{B}^{\vee}\hspace{.5cm}}&\PP(\mathrm{Sym}^{2}B^{\vee})}
\]
of both of the spaces of conics $\PP(\mathrm{Sym}^{2}B)$ and $\PP(\mathrm{Sym}^{2}B^\vee)$ along $\PP (B)$ and $\PP (B^{\vee})$ the locus of conics of rank one. Then one can obtain the following description of $\tilde{D}_1$ and $D_2\setminus D_1$. 

\begin{proposition}[\cite{KL, OGr99}]\label{description of exceptional divisors in M_2}
\leavevmode

\begin{itemize}
    \item [(i)] $\tilde{D}_{1}$ is the $\mathbf{C}(\Ss_3)$-bundle over $J$. That is, $\tilde{D}_{1}$ is the $(\hat{\PP}^{5})$-bundle over $\mathrm{Gr}(3,g)$ over $J$, where $\hat{\PP}^{5}$ is the blow-up of $\PP^{5}$ along $\PP^{2}$.
    \item [(ii)] $D_{2}\setminus D_{1}$ is the quotient of a $(\PP^{g-2}\times\PP^{g-2})$-bundle over $(J\times J)\setminus J$ by the action of $\ZZ_{2}$.
\end{itemize}
\end{proposition}

\begin{theorem}[Section 6.2 of \cite{KL}]\label{vvp of moduli of vector bundles of rk 2 and deg d}
The virtual Poincar\'{e} polynomial of $\mathbf{M}_{D}(2,d)$ for even $d$, is given by the formula
\begin{equation}\label{dg}
D_g(t):=\mathrm{P}(\mathbf{M}_{D}(2,d))=\frac{(1+t)^{2g}((1+t^{3})^{2g}-t^{2g+2}(1+t)^{2g})}{(1-t^{2})(1-t^{4})}-\frac{t^{2}}{2}\left(\frac{(1+t)^{4g}}{1-t^{2}}-\frac{(1-t^{2})^{2g}}{1+t^{2}}\right).
\end{equation}
\end{theorem}
\begin{proof}
The virtual Poincar\'{e} polynomial of $\mathbf{M}_{D}(2,d)$ is
\begin{align*}
\mathrm{P}(\mathbf{M}_{D}(2,d))=\mathrm{P}(\mathbf{M}_{D}(2,0))&=\mathrm{P}(\mathbf{M}_{D}(2,0)^{s})+\mathrm{P}((J\times J)/\ZZ_{2})\\
&=\mathrm{P}(\mathbf{M}_{2})-\mathrm{P}(D_{1})-\mathrm{P}(D_{2}\setminus D_{1})+\mathrm{P(}(J\times J)/\ZZ_{2}).
\end{align*}
By Proposition \ref{description of exceptional divisors in M_2}, \cite[Proof of Lemma 4.8]{K} and \cite[Section 6.2]{KL},
\begin{align*}
\mathrm{P}(\tilde{D}_{1})&=\mathrm{P}(\hat{\PP}^{5})\cdot\mathrm{P}(\mathrm{Gr}(3,g))\cdot\mathrm{P}(J)=(\mathrm{P}(\PP^{5})+\mathrm{P}(\PP^{2}\times\PP^{2})-\mathrm{P}(\PP^{2}))\cdot\mathrm{P}(\mathrm{Gr}(3,g))\cdot\mathrm{P}(J)\\
&=\left(\frac{1-t^{12}}{1-t^{2}}+\left(\frac{1-t^{6}}{1-t^{2}}\right)^{2}-\frac{1-t^{6}}{1-t^{2}}\right)\frac{(1-t^{2g})(1-t^{2g-2})(1-t^{2g-4})}{(1-t^{2})(1-t^{4})(1-t^{6})}(1+t)^{2g}\\
&=\frac{1-t^{8}}{(1-t^{2})^{2}}\frac{(1-t^{2g})(1-t^{2g-2})(1-t^{2g-4})}{(1-t^{2})(1-t^{4})}(1+t)^{2g}
\end{align*}
By the proof of \cite[Lemma 3.5.4]{OGr97}, we see that the exceptional divisor of $\tilde{D}_{1}$ is $(\PP^{g-3})$-bundle over $(\PP(S^{2}\Ss_2))$-bundle over $J$. By the usual blowing-up formula of \cite[p.605]{GH78},
\begin{align*}
\mathrm{P}(D_{1})&=\mathrm{P}(\tilde{D}_{1})-\mathrm{P}(\PP(S^{2}\Ss_2))\cdot\mathrm{P}(J)\cdot(\mathrm{P}(\PP^{g-3})-1)\\
&=\left(\frac{1-t^{8}}{(1-t^{2})^{2}}\frac{(1-t^{2g})(1-t^{2g-2})(1-t^{2g-4})}{(1-t^{2})(1-t^{4})}-\frac{1-t^{6}}{1-t^{2}}\frac{(1-t^{2g})(1-t^{2g-2})}{(1-t^{2})(1-t^{4})}\frac{t^{2}(1-t^{2(g-3)})}{1-t^{2}}\right)(1+t)^{2g}\\
&=\frac{(1+t)^{2g}(1-t^{2g})(1-t^{2g-2})}{(1-t^{2})^{3}(1-t^{4})}\left((1-t^{8})(1-t^{2g-4})-(1-t^{6})(t^{2}-t^{2g-4})\right)\\
&=\frac{(1+t)^{2g}(1-t^{2g})(1-t^{2g-2})}{(1-t^{2})^{3}(1-t^{4})}\left((1+t^{2g+4})-(t^{2}+t^{2g+2})\right)
\end{align*}
Further, we have
\begin{align*}
\mathrm{P}(D_{2}\setminus D_{1})=&\mathrm{P}(\PP^{g-2}\times\PP^{g-2})^{+}\cdot\mathrm{P}((J\times J)\setminus J)^{+}+\mathrm{P}(\PP^{g-2}\times\PP^{g-2})^{-}\cdot \mathrm{P}((J\times J)\setminus J)^{-}\\
=&\mathrm{P}(\PP^{g-2}\times\PP^{g-2})^{+}\left(\frac{1}{2}\mathrm{P}_{t}(J)^{2}+\frac{1}{2}\mathrm{P}_{-t^{2}}(J)-\mathrm{P}(J)\right)+\mathrm{P}(\PP^{g-2}\times\PP^{g-2})^{-}\left(\frac{1}{2}\mathrm{P}_{t}(J)^{2}-\frac{1}{2}\mathrm{P}_{-t^{2}}(J)\right)\\
=&\frac{(1-t^{2g})(1-t^{2g-2})}{(1-t^{2})(1-t^{4})}\left(\frac{1}{2}(1+t)^{4g}+\frac{1}{2}(1-t^{2})^{2g}-(1+t)^{2g}\right) \\
&+t^{2}\frac{(1-t^{2g-2})(1-t^{2g-4})}{(1-t^{2})(1-t^{4})}\left(\frac{1}{2}(1+t)^{4g}-\frac{1}{2}(1-t^{2})^{2g}\right),
\end{align*}
where $\mathrm{P}(Z)^{+}$ (resp. $\mathrm{P}(Z)^{-}$) denotes the virtual Poincar\'{e} polynomial of the $\ZZ_{2}$-invariant (resp. anti-invariant) part of $\mathrm{H}_{c}^{*}(Z)$. By \cite[Proof of Lemma 4.8]{K}, we see that
\[
\mathrm{P}((J\times J)/\ZZ_{2})=\mathrm{P}(J \times J)^{+}=\frac{1}{2}\mathrm{P}_{t}(J)^{2}+\frac{1}{2}\mathrm{P}_{-t^{2}}(J)=\frac{1}{2}(1+t)^{4g}+\frac{1}{2}(1-t^{2})^{2g}.
\]
Further, except for the first term $\displaystyle\frac{(1+t)^{2g}((1+t^{3})^{2g}-t^{2g+2}(1+t)^{2g})}{(1-t^{2})(1-t^{4})}$, we can simplify the sum of the terms containing $(1+t)^{4g}$ and $(1-t^{2})^{2g}$ into $\displaystyle-\frac{t^{2}}{2}\left(\frac{(1+t)^{4g}}{1-t^{2}}-\frac{(1-t^{2})^{2g}}{1+t^{2}}\right)$. Hence, we get the formula of the virtual Poincar\'{e} polynomial of $\mathbf{M}_{D}(2,d)$ as follows:
\begin{align*}
\mathrm{P}(\mathbf{M}_{D}(2,d))=&\frac{(1+t)^{2g}((1+t^{3})^{2g}-t^{2g+2}(1+t)^{2g})}{(1-t^{2})(1-t^{4})}-\frac{t^{2}}{2}\left(\frac{(1+t)^{4g}}{1-t^{2}}-\frac{(1-t^{2})^{2g}}{1+t^{2}}\right)\\
&+\frac{(1+t)^{2g}}{1-t^{2}}\left\{\frac{t^{2}(1-t^{6g-2})}{1-t^{4}}-\frac{t^{4g-2}(1-t^{2g})}{1-t^{2}}+\frac{t^{4}(1-t^{4g-6})(1-t^{2g-2})}{(1-t^{2})(1-t^{4})}-t^{2g}\left(\frac{1-t^{2g-2}}{1-t^{2}}\right)^{2}\right\}\\
&+\frac{(1+t)^{2g}(1-t^{2g})(1-t^{2g-2})}{(1-t^{2})(1-t^{4})}\left(1-\frac{(1+t^{2g+4})-(t^{2}+t^{2g+2})}{(1-t^{2})^{2}}\right).
\end{align*}
Finally, it is easy to see that the second and third lines of the above formula of $\mathrm{P}(\mathbf{M}_{D}(2,d))$
\begin{align*}
\frac{(1+t)^{2g}}{1-t^{2}}\left\{\frac{t^{2}(1-t^{6g-2})}{1-t^{4}}-\frac{t^{4g-2}(1-t^{2g})}{1-t^{2}}+\frac{t^{4}(1-t^{4g-6})(1-t^{2g-2})}{(1-t^{2})(1-t^{4})}-t^{2g}\left(\frac{1-t^{2g-2}}{1-t^{2}}\right)^{2}\right\}\\
+\frac{(1+t)^{2g}(1-t^{2g})(1-t^{2g-2})}{(1-t^{2})(1-t^{4})}\left(1-\frac{(1+t^{2g+4})-(t^{2}+t^{2g+2})}{(1-t^{2})^{2}}\right)
\end{align*}
is zero.
\end{proof}

\begin{remark}
In the proof of Theorem \ref{vvp of moduli of vector bundles of rk 2 and deg d}, we compute $\mathrm{P}(\mathbf{M}_{D}(2,0)^{s})=\mathrm{P}(\mathbf{M}_{2})-\mathrm{P}(D_{1})-\mathrm{P}(D_{2}\setminus D_{1})$ directly. However, it was already done in Section 6.2 of \cite{KL}.
\end{remark}

\begin{remark}\label{reem2}
Setting $\mathbf{M}_{D}(2,\Ll)$ as the fibre of the surjection $\mathbf{M}_D(2,d) \rightarrow \Pic^d(D)$ over $[\Ll]$, we have
\begin{align*}
\mathrm{P}(\mathbf{M}_D(2,d))&=\mathrm{P}(\Pic^d(D))\cdot \mathrm{P}(\mathbf{M}_D(2,\Ll) )\\
&=
\begin{cases}
(1+t)^{2g} \cdot (1+t)^{2g-2} \cdot S_g(t), & \text{if } d \; \text{is odd};\\
(1+t)^{2g} \cdot U_g(t), & \text{if } d \; \text{is even},
\end{cases}
\end{align*}
where $S_g(t)$ and $U_g(t)$ are polynomials with integer coefficients.
\end{remark}

\begin{proof}[Proof of Theorem \ref{virtu}]
The projective bundle $\PP(\EE)$ and $\mathbf{Q}$ have the virtual Poincar\'e polynomials 
\[
\mathrm{P}(\PP(\EE))=\mathrm{P}(\mathbf{M}_{C_1}(2,1))\cdot \mathrm{P}(\mathbf{M}_{C_2}(2,1)) \cdot \mathrm{P}(\PP^3), \quad \mathrm{P}(\mathbf{Q})=\mathrm{P}(\mathbf{M}_{C_1}(2,1))\cdot \mathrm{P}(\mathbf{M}_{C_2}(2,1)) \cdot \mathrm{P}(\PP^1 \times \PP^1),
\]
where $\mathrm{P}(\PP^3)=1+t^2+t^4+t^6$ and $\mathrm{P}(\PP^1\times \PP^1)=(1+t^2)^2$. By Theorem \ref{main}, we obtain
\[
\mathrm{P}(\mathbf{U}_C(\mathbf{w}, 2, 0)\setminus \mathbf{U}_C(\mathbf{w}, 2, 0)_{(0,0)})=\mathrm{P}(\PP (\EE)\setminus \mathbf{Q})=\mathrm{P}(\PP (\EE))-\mathrm{P}(\mathbf{Q}).
\]
We also have $\mathrm{P}(\mathbf{U}_C(\mathbf{w}, 2, 0)_{(0,0)})=\mathrm{P}(\mathbf{M}_{C_1}(2,0)\times \mathbf{M}_{C_2}(2,0))$. Consequently, the virtual Poincar\'e polynomial of $\mathbf{U}_C(\mathbf{w}, 2, 0)$ is given by
\[
\mathrm{P}(\mathbf{U}_C(\mathbf{w}, 2, 0))=\mathrm{P}(\mathbf{M}_{C_1}(2,2g_1-1))\cdot \mathrm{P}(\mathbf{M}_{C_2}(2,2g_2-1))(-t^2+t^6)\\+\mathrm{P}(\mathbf{M}_{C_1}(2,2g_1-2))\cdot \mathrm{P}(\mathbf{M}_{C_2}(2,2g_2-2)).
\]
Now, using the notations in (\ref{cg}) and (\ref{dg}), $\mathrm{P}(\mathbf{U}_C(\mathbf{w},2,0))$ can be written as 
\[
\mathrm{P}(\mathbf{U}_C(\mathbf{w},2,0))=C_{g_1}(t)\,C_{g_2}(t)\,(-t^2+t^6)+D_{g_1}(t)\,D_{g_2}(t).
\]
Then the assertion follows from Remark \ref{reem2}. 
\end{proof}

\begin{corollary}\label{comparison}
For a smooth curve $C'$ of genus $g$ over the complex number field, we have
\[
\mathrm{P}(\mathbf{U}_C(\mathbf{w}, 2, 0)) = \mathrm{P}(\mathbf{M}_{C'}(2, 2g-2))
\]
if and only if $g=2$.
\end{corollary}

\begin{proof}
Using the notations in (\ref{cg}) and (\ref{dg}), the difference $F(t):=\mathrm{P}(\mathbf{U}_C(\mathbf{w},2,0))-\mathrm{P}(\mathbf{M}_{C'}(2,2g-2))$ can be written as 
\begin{equation}\label{expp}
F(t)=C_{g_1}(t)\,C_{g_2}(t)\,(-t^2+t^6)+D_{g_1}(t)\,D_{g_2}(t)-D_{g}(t).
\end{equation}
One can check that $F(t)\equiv 0$ when $(g_1, g_2)=(1,1)$ by a direct computation. Indeed, we have
\[
\mathrm{P}(\mathbf{U}_C(\mathbf{w}, 2, 0))=\mathrm{P}(\mathbf{M}_{C'}(2, 2))=(1+t)^4(t^6+t^4+t^2+1).
\]
For the converse, it would be enough to show that $F(1)=D_{g_1}(1)D_{g_2}(1) - D_{g}(1) \ne 0$ for $g\ge 3$. Observe that
\[
(1-t+t^2)^g - t^{g+1} = (1-t) Q_g(t), \qquad
(1+t)^{2g-1}(1+t^2) - (1-t)^{2g+1} = 2 t \tilde{Q}_g(t),
\]
where $Q_g(t), \tilde{Q}_g(t) \in \ZZ[t]$. This leads to:
\begin{align*}
D_g(1) 
&= \lim_{t \to 1} D_g(t) \\
&= \lim_{t \to 1} \frac{(1+t)^{2g+1}}{1-t^{4}}\cdot \big((1+t)^{2g-2}Q_g(t)\cdot((1-t+t^2)^g+t^{g+1})-t^3\tilde{Q}_g(t) \big) \\
&= \lim_{t \to 1} \frac{ \frac{d}{dt} \Big[ (1+t)^{2g+1} \big( (1+t)^{2g-2}Q_g(t)\cdot((1-t+t^2)^g+t^{g+1})-t^3\tilde{Q}_g(t) \big) \Big] }
{ \frac{d}{dt} (1-t^4) }\\
&= \frac{ 2^{2g+1} \left((2g-2)2^{2g-2}Q_g(1)+2^{2g-1}Q_g'(1)+2^{2g-2}Q_g(1)(2g+1)-3\tilde{Q}_g(1)-\tilde{Q}_g'(1) \right) }{ -4 } \\
&= 2^{4g-2}(3-g).
\end{align*}
Therefore, we have 
\[
D_{g_1}(1)D_{g_2}(1)-D_{g_1+g_2}(1)=2^{4g-4}(g+g_1g_2-3),
\]  
and this is not zero for $g\ge 3$. 
\end{proof}

\begin{remark}\label{Poin:ell}
In case $(g_1, g_2)=(1,1)$, the moduli space $\mathbf{U}_C(\mathbf{w}, 2, 0)$ is isomorphic to $\PP(\EE)$, and so it is smooth. The virtual Poincar\'e polynomial is identical to the classical Poincar\'e polynomial, and it is given by 
\[
\mathrm{P}(C_1){\cdot}\mathrm{P}(C_2){\cdot}\mathrm{P}(\PP^3)=(1+2t+t^2)^2(1+t^2+t^4+t^6), 
\]
because $\mathbf{M}_{C_i}(2,1)\cong C_i$ and $\mathbf{M}_{C_i}(2,0)\cong \mathrm{Sym}^2(C_i)$ is $(\PP^1)$-bundle over $C_i$ for each $i$. Assume that each $g_i\ge 2$. Under the map $\Phi : \PP(\EE) \rightarrow \mathbf{U}_C(\mathbf{w}, 2, 0)$, the subvariety $\mathbf{Q}$ contracts to $\mathbf{U}_C(\mathbf{w}, 2, 0)_{(0,0)} \cong \mathbf{M}_{C_1}(2,2g_1-2)\times \mathbf{M}_{C_2}(2,2g_2-2)$ along the family of nontrivial extensions (\ref{eerrr1}) in Proposition \ref{(2,0)(0,1)}. 
Lastly, without loss of generality, assume that $g_1=1$. The fibre $\PP^1 \times \PP^1$ in $\mathbf{Q}$ over $(\Ee_1, \Ee_2)\in \mathbf{M}_{C_1}(2,1)\times \mathbf{M}_{C_2}(2,2g_2-1)$ maps to $\mathbf{M}_{C_1}(2,0)\times \mathbf{M}_{C_2}(2,2g_2-2)$ isomorphically. Indeed, the first ruling of $\PP^1\times \PP^1$ maps to $\mathbf{M}_{C_1}(2,0) \cong \mathrm{Sym^2}(C_1)$ isomorphically, as explained in Remark \ref{QQ}, while the second ruling maps to $\mathbf{M}_{C_2}(2,2g_2-2)$ isomorphically as a conic, called a {\it Hecke curve}; see \cite{narashimhan-ramanan1}. 
\end{remark}

\begin{remark}
The polynomial $F(t)=\mathrm{P}(\mathbf{U}_C(\mathbf{w},2,0))-\mathrm{P}(\mathbf{M}_{C'}(2,2g-2))$ depends both on $g_1$ and $g_2$, not just on $g$. For instance, an explicit computation when $g=4$ shows that
\begin{equation}\label{g_1g_2}
F(t) =
\begin{cases}
t^5 (1+t)^8 \, (t^4-5) (t^7-5t^3-8t^2-8), & \text{for } (g_1,g_2)=(2,2),\\[1ex]
t^5 (1+t)^8 \, \Big(-2(t^{10}+t^8+6t^7+t^6+t^4-14t^3-14t^2-14)\Big), & \text{for } (g_1,g_2)=(1,3).
\end{cases}
\end{equation}
One can also observe that $F(t)$ always factors as
\[
F(t)=t^{5}(1+t)^{2g}H_{g_1,g_2}(t),
\]
for some polynomial $H_{g_1,g_2}(t) \in \ZZ[t]$ with $H_{g_1, g_2}(0)\ne 0$. Indeed, for $g\ge 2$ we have 
\[
C_g(t)=1+2g\,t+\Bigl(\tbinom{2g}{2}+1\Bigr)t^{2}
+\Bigl(\tbinom{2g}{3}+4g\Bigr)t^{3}
+\Bigl(\tbinom{2g}{4}+\tbinom{2g}{2}+4g^{2}+(a_g+2)\Bigr)t^{4}
+O(t^{5}),
\]
with $a_2=-1$ and $a_g=0$ for $g\ge 3$. Similarly, the expansion of $D_g(t)$ is
\[
D_g(t)=1+2g\,t+(2g^{2}-g+1)t^{2}
+\frac{2g(2g^{2}-3g+4)}{3}\,t^{3}
+\Bigl(\tfrac{2}{3}g^{4}-2g^{3}+\tfrac{23}{6}g^{2}-\tfrac{3}{2}g+1\Bigr)t^{4}
+O(t^{5}).
\]
By a direct computation with (\ref{expp}), we obtain
\[
F(t)\equiv 0\pmod{t^5}. 
\]
Furthermore, observe that the coefficient of the $t^{5}$-term in $D_{g}(t)$ is
\[
\binom{2g}{5}+\binom{2g}{3}+4g+2g^{2}(2g-1)-\frac{g(4g-1)(4g-2)}{3}.
\]
One can check that the coefficient of the $t^5$-term in $F(t)$ is $4g_1g_2(g-1)-2g$, and so it is positive for $g\ge 3$. Then the assertion follows from Gauss's lemma. 
\end{remark}

\end{document}